\documentclass[letterpaper]{article}
\usepackage{tgpagella}
\usepackage{mathpazo}
\usepackage{graphicx} 
\usepackage[english]{babel}
\usepackage{amsfonts,amsthm,amsmath,amssymb,mathtools,nicefrac}
\usepackage{fullpage}
\usepackage{authblk}
\usepackage[shortlabels]{enumitem}
\usepackage{todonotes}
\usepackage{hyperref}
\usepackage{bbm,stmaryrd}
\usepackage[most]{tcolorbox}
\usepackage{dsfont}
\usepackage{tikz}
\usetikzlibrary{calc}

\def\rmlabel{\upshape({\kern+0.0833em\itshape\roman*\kern+0.0833em})}
\def\nlabel{\upshape(\kern+0.0833em{\itshape\arabic*}\kern0.0833em)}
\def\alabel{\upshape({\kern+0.0833em\itshape\alph*\kern0.0833em})}
\def\Alabel{\upshape({\kern+0.0833em\itshape\Alph*\kern0.1111em})}

\newcommand{\cA}{\mathcal{A}}
\newcommand{\cB}{\mathcal{B}}
\newcommand{\cC}{\mathcal{C}}
\newcommand{\cR}{\mathcal{R}}
\newcommand{\cP}{\mathcal{P}}
\newcommand{\cE}{\mathcal{E}}

\usepackage{imakeidx}
\newcommand{\By}[2]{\overset{\mbox{\tiny{#1}}}{#2}}
\newcommand{\ByRef}[2]{   \By{\eqref{#1}}{#2} }

\newcommand{\lBy}[1]{     \By{#1}{<} }
\newcommand{\gBy}[1]{     \By{#1}{>} }
\newcommand{\leBy}[1]{    \By{#1}{\le} }
\newcommand{\geBy}[1]{    \By{#1}{\ge} }
\newcommand{\eqByRef}[1]{ \ByRef{#1}{=} }
\newcommand{\lByRef}[1]{  \ByRef{#1}{<} }

\newcommand{\leByRef}[1]{ \ByRef{#1}{\le} }
\newcommand{\geByRef}[1]{ \ByRef{#1}{\ge} }
\newcommand{\JUSTIFY}[1]{\fbox{\tiny{#1}}\quad}

\newtheorem{lem}{Lemma}[section]		
\newtheorem{fact}[lem]{Fact}		
		
\newtheorem{defi}[lem]{Definition}

\newtheorem{thm}[lem]{Theorem}
		
\newtheorem{conj}[lem]{Conjecture}		
\newtheorem{prop}[lem]{Proposition}

\newcounter{claim}
\renewcommand{\theclaim}{\Alph{claim}}

\newenvironment{claim}{
  \refstepcounter{claim}
  \par\medskip
  \noindent\textbf{Claim \theclaim.}\itshape
}{
  \par\medskip
}

\newcounter{proofctr}

\pretocmd{\proof}{%
  \stepcounter{proofctr}%
  \setcounter{claim}{0}%
}{}{}

\newenvironment{claimproof}{
  \begingroup
  \let\origqed\qedsymbol
  \renewcommand{\qedsymbol}{$\blacktriangleleft$}
  \par\noindent\textit{Proof of Claim \theclaim.}
}{
  \qed
  \par\medskip
  \let\qedsymbol\origqed
  \endgroup
}

\newcommand{\G}{\mathbb{G}}
\newcommand{\N}{\mathbb{N}}

\newcommand{\R}{\mathbb{R}}

\newcommand{\Probability}{\mathbf{P}} 
\newcommand{\Expectation}{\mathbf{E}} 
\newcommand{\Variance}{\mathbf{Var}}
\newcommand{\diff}{\mathsf{d}}
\newcommand{\eps}{\varepsilon}
\newcommand{\epsabs}{\eps_{\mathrm{abs}}}

\def\Kcyc{K^{\text{cyc}}}

\def\cQ{\mathcal Q}
\def\SMALLDEG{\mathsf{D}}

\DeclareMathOperator{\wdeg}{wdeg}

\DeclareMathOperator{\supp}{supp}

\newcommand{\cutn}[1]{\left\lVert #1\right\rVert_{\square}}

\newcommand{\eqtext}[1]{\parbox{0.86\linewidth}{#1}}

\title{Hamiltonicity of inhomogeneous random graphs\thanks{Research supported by Czech Science Foundation Project 26-23695S.}}
\author[1]{Frederik Garbe\thanks{Email: garbe@cs.cas.cz}}
\author[1]{Jan Hladký\thanks{Email: hladky@cs.cas.cz}}
\author[2]{Simón Piga\thanks{Email: sipiga@ibs.re.kr. This work was done while at the Institute of Computer Science of the Czech Academy of Sciences.}}
\date{}
\affil[1]{Institute of Computer Science of the Czech Academy of Sciences\\
Pod Vodárenskou věží 2, 182 00 Praha 8, Czechia. With institutional support RVO:67985807.}
\affil[2]{Extremal Combinatorics and Probability Group, Institute for Basic Science, Daejeon, South Korea}
\begin{document}
\maketitle
\begin{abstract}
We provide a complete characterization of those graphons $W$ for which the inhomogeneous random graph \(\G(n,W)\) is asymptotically almost surely Hamiltonian. The characterization involves three conditions. Two of them constitute the characterization of $\G(n,W)$ being a.a.s.\ connected, as was shown recently by Hladký and Viswanathan. The third condition captures a geometric obstacle which prevents $\G(n,W)$ from having perfect fractional matchings. 
A prominent feature of the positive direction of our proof is the use of weak* limits to translate the absence of the above geometric obstacle into favourable matching properties of a regularization of $\G(n,W)$.
\end{abstract}

\section{Introduction}

The study of Hamilton cycles, i.e., cycles that traverse every vertex of a graph exactly once, is one of the central themes in graph theory. In the context of random graphs, this inquiry dates back to the very foundations of the field, laid by Gilbert~\cite{Gilbert59} and by Erdős and Rényi~\cite{ErdosRenyi59} in 1959. The model introduced by Gilbert, denoted as $\G(n,p)$ (where each edge appears independently with probability $p$), has served as the testbed for understanding the emergence of Hamiltonicity.

Nowadays, the emergence of global spanning structures in this homogeneous model is well understood. The classical results of Erdős and Rényi~\cite{ErdosRenyi59}, obtained for the closely related model $\G(n,M)$ in which $M$ edges are chosen uniformly at random, established the sharp threshold for connectivity; in the $\G(n,p)$ parametrization, this threshold is $p =(1+o(1))\frac{\ln n}{n}$. Specifically, for $p = \frac{\ln n + c}{n}$ with a fixed $c\in\R$, the probability that $\G(n,p)$ is connected converges to $e^{-e^{-c}}$, a limit determined solely by the disappearance of isolated vertices. The problem of Hamiltonicity, while significantly harder, exhibits a striking similarity. A famous result of Pósa~\cite{POSA1976359} determined that the threshold for Hamiltonicity is $p=\Theta(\frac{\ln n}{n})$, using the `rotation technique' that turned out to be very versatile. Independently and at the same time, a much more precise bound on the threshold $p=(1+o(1))\frac{\ln n}{n}$ was obtained by Korshunov~\cite{Kor76}. This was subsequently refined by Komlós and Szemerédi~\cite{KomlosSzemeredi83}. Later, celebrated hitting-time results unified these observations, showing that in the random graph process, in which edges are added one-by-one and at random starting with the edgeless graph, the obstructions to both properties are purely local: a.a.s.\ the graph becomes connected the moment the last isolated vertex disappears, as was proved by Bollobás and Thomason~\cite{BollobasThomason85}, and it becomes Hamiltonian exactly when the minimum degree reaches~2, as was proved independently by Bollobás~\cite{Bollobas84} and by Ajtai, Komlós, and Szemerédi~\cite{AjKoSz85}. For more details, we refer to the relevant chapters in classical monographs (Section~6.2 in~\cite{MR3675279} and Section~8.2 in~\cite{MR1864966}) on random graphs, as well as a continuously updated survey~\cite{frieze2025hamiltoncyclesrandomgraphs}.

However, $\G(n,p)$ assumes homogeneity: every vertex plays the same probabilistic role. This assumption often fails to capture the complexity of real-world networks or more sophisticated mathematical structures, which display intrinsic inhomogeneities such as community structure or heavy-tailed degree distributions. To study such phenomena, we turn to the theory of graph limits, specifically \emph{graphons}.

A graphon is a symmetric measurable function $W: \Omega^2 \to [0,1]$, where $(\Omega, \mu)$ is an atomless standard probability space. Throughout the paper, all the sets and all the functions we work with are tacitly assumed to be measurable. Graphons emerged in the works of Lovász and Szegedy~\cite{LOVASZ2006933} and of Borgs, Chayes, Lovász, Sós, and Vesztergombi~\cite{BORGS20081801} as the natural limit objects for sequences of dense graphs. Just as every convergent sequence of graphs yields a limit graphon, every graphon $W$ defines a natural model of inhomogeneous random graphs, denoted $\G(n,W)$.

The generation of a random graph $\G(n,W)$ proceeds in two distinct steps, having fixed the vertex set as~$V=[n]$:
\begin{enumerate}[label=(GS\arabic*)]
    \item\label{G1} \textbf{Generate vertex types:} We sample $n$ independent points $X_1, \dots, X_n$ from $\Omega$ according to the distribution $\mu$. These points are effectively the ``types'' or ``hidden variables'' associated with the vertices $1, \dots, n$.
    \item\label{G2} \textbf{Generate edges:} Conditioned on the types $X_1, \dots, X_n$, we form the graph on the vertex set $V$ by including the edge $\{i,j\}$ with probability $W(X_i, X_j)$, independently for all pairs $1 \le i < j \le n$.
\end{enumerate}
When $W(x,y) \equiv p$ is constant, this recovers $\G(n,p)$. However, general $W$'s allow for rich structural variations. For instance, the Stochastic Block Model corresponds to $W$ being a step function, i.e., there is a partition $\Omega=\bigsqcup_{i\in[\ell]}S_i$ such that $W|_{S_i\times S_j}$ is constant for all $i,j\in[\ell]$.

In a recent paper~\cite{ConnectivityPaper} the second author and Viswanathan analyzed the connectivity of $\G(n,W)$. The main result of~\cite{ConnectivityPaper} is that $\G(n,W)$ is a.a.s.\ connected if and only if the following two conditions are satisfied for $W$:
\begin{enumerate}[label=(C\arabic*)]
    \item\label{C1} The graphon $W$ is itself \emph{connected}, meaning that whenever $\Omega$ is decomposed into two sets $\Omega=S\sqcup T$ with $\mu(S), \mu(T) > 0$ then $\int_{S\times T}W>0$.
    \item\label{C2} The tail of points of small degree in $W$ is light. Specifically, we define the degree of a point $x \in \Omega$ as $\deg_W(x) = \int_{\Omega} W(x,y) \diff\mu(y)$ and for $\alpha\in [0,1]$ we define $\SMALLDEG_W(\alpha):=\{x\in \Omega\colon\deg_W(x)\le \alpha\}$. Then we require that $\lim_{\alpha \searrow 0} \frac{\mu(\SMALLDEG_W(\alpha))}{\alpha} = 0$.
\end{enumerate}
Condition~\ref{C1} is necessary for the absence of macroscopic disconnections, that is, partitions $V=V_1\sqcup V_2$ with $|V_1|,|V_2|=\Theta(n)$ and no edges between $V_1$ and $V_2$. Condition~\ref{C2} is necessary for the absence of isolated vertices. 

In this paper, we address the much stronger property of Hamiltonicity in $\G(n,W)$. One might hope that, as in $\G(n,p)$, the conditions for connectivity~\ref{C1} and~\ref{C2} suffice for Hamiltonicity. As it turns out, inhomogeneity introduces a third, purely structural obstacle. More precisely, conditions~\ref{C1} and~\ref{C2} alone do not guarantee that $\G(n,W)$ has a perfect fractional matching, which is clearly necessary for Hamiltonicity. Using an LP duality perspective, we can distill the obstacle to $\G(n,W)$ having a perfect fractional matching into a notion we call a `peninsula'. We motivate and define this notion in the next subsection. Our main theorem, Theorem~\ref{thm:mainHC}, then says that the absence of peninsulae plus the connectivity conditions~\ref{C1} and~\ref{C2} indeed yield that $\G(n,W)$ is a.a.s.\ Hamiltonian.

\subsection{Peninsulae}\label{sssec:peninsula}
We start by defining graphon peninsulae and then motivate the definition by looking at the more intuitive setting of finite graphs.
\begin{defi}\label{defi:peninsula}
Suppose that we have a graphon $W:\Omega^2\to[0,1]$.
\begin{enumerate}[label=(\roman*)]
    \item\label{en:defiPen1} $W$ has a \emph{peninsula} if there exists a number $a\in(0,\frac12]$ and disjoint sets $A,B\subseteq \Omega$ such that $\mu(A)= a$, $\mu(B)=1-2a$, and $W$ is constant-0 on $A\times(A\cup B)$. 
    \item\label{en:defiPen2} $W$ has a \emph{narrow peninsula} if there exists a number $a\in(0,\frac12]$ and disjoint sets $A,B\subseteq \Omega$ such that $\mu(A)>a$, $\mu(B)=1-2a$, and $W$ is constant-0 on $A\times(A\cup B)$.
\end{enumerate}
\end{defi}
The inspiration for the notions of peninsula and narrow peninsula stems from the matching theory for finite graphs. In this context, given an $n$-vertex graph $G$, we analogously say that for $a\in (0,\frac12n]$ with $2a\in\N$, a pair of disjoint sets $A,B\subseteq V(G)$ with $|A|\ge a$ and $|B|=n-2a$ is a \emph{graph peninsula} if $G$ contains no edge with one end-vertex in $A$ and the other in $A\cup B$. Note that in Definition~\ref{defi:peninsula}\ref{en:defiPen1} we could equally have required $\mu(A)\ge a$, since $\Omega$ is atomless and $A$ can therefore be shrunk to measure exactly~$a$; in the finite setting we ask for $|A|\ge a$ because $a$ need not be an integer. 
\begin{figure}[t]
\centering
	\includegraphics[scale=0.7]{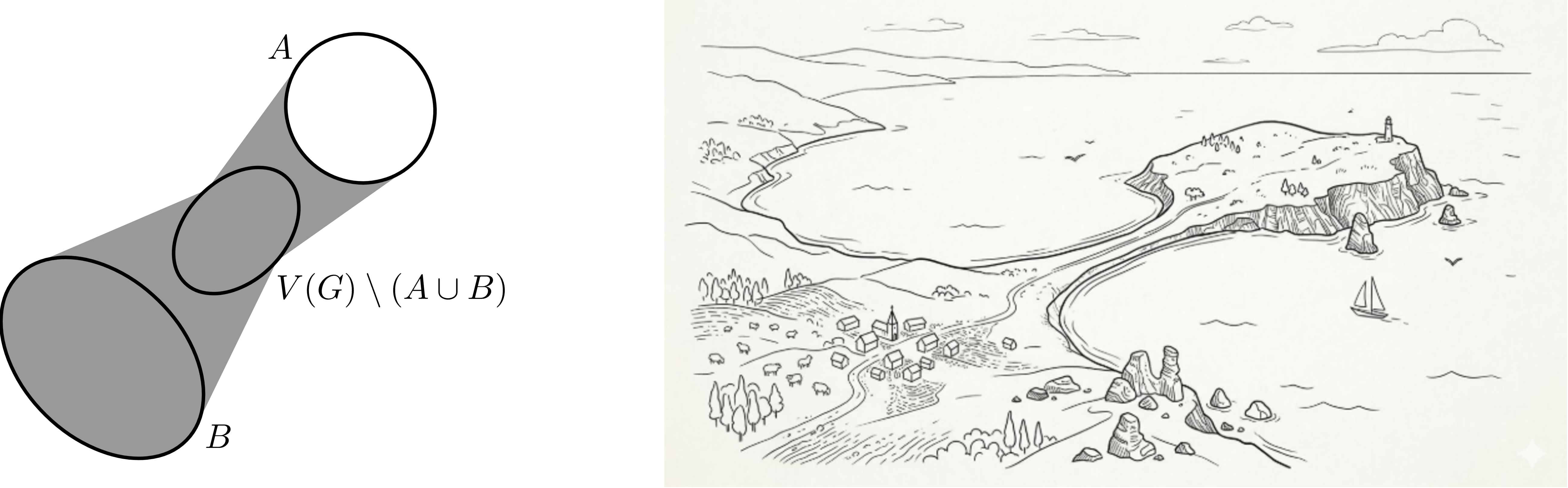}
\caption{\emph{Left:} A narrow graph peninsula; dark regions indicate admissible edges.
\emph{Right:} A beautiful peninsula. Included for pleasure.}
	\label{fig:peninsula}
\end{figure}
Further, a \emph{narrow graph peninsula} amounts to the same condition but with $|A|>a$ and $|B|=n-2a$. See Figure~\ref{fig:peninsula} for an example. Let us explain how this notion connects to matching theory. Recall that a function $f:V(G)\to[0,1]$ is a \emph{fractional vertex cover of $G$} if for every edge $uv$ of $G$ we have $f(u)+f(v)\ge 1$. A fractional vertex cover is \emph{half-integral} if $f:V(G)\to\{0,\frac12,1\}$. For each half-integral vertex cover $f$, we can take $A:=f^{-1}(0)$ and $B:=f^{-1}(\frac12)$. We then have the following correspondence:
\begin{enumerate}[label=(P\arabic*)]
    \item\label{corr1} a graph peninsula corresponds to a half-integral vertex cover $f$ with total weight $\| f\|_1\le \frac{n}{2}$ which is not constant-$\frac12$,
    \item\label{corr2} a narrow graph peninsula corresponds to a half-integral vertex cover $f$ with total weight $\| f\|_1<\frac{n}{2}$.
\end{enumerate}
\begin{figure}[t]
\centering
	\includegraphics[scale=0.7]{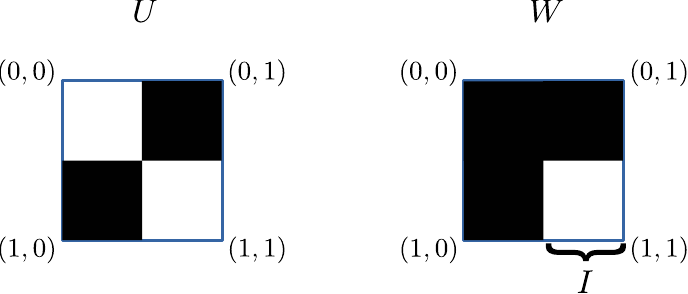}
	\caption{Graphons $U$ and $W$ from Section~\ref{sssec:peninsula}}
	\label{fig:bipgraphons}
\end{figure}
By the LP duality, a narrow graph peninsula means that $G$ does not have a perfect fractional matching (and hence not a Hamilton cycle). Actually, this is the easy direction of the LP duality and can be seen directly: Consider an arbitrary fractional matching. Each edge in its support with one end in $A$ must have the other end in $V(G)\setminus (A\cup B)$. Hence, the total weight of such a matching on the vertices of $A$ is at most the total weight on $V(G)\setminus (A\cup B)$. As $|A|>|V(G)\setminus (A\cup B)|$, we see that the matching does not cover $A$, and hence is not perfect.

The meaning of graph peninsulae is more subtle. Indeed, illustrative examples of graphs which do have a graph peninsula are either the complete balanced bipartite graph $H=K_{n/2,n/2}$, or a graph $G$ which is obtained from $H$ by adding edges of $K_{n/2}$ into one of the bipartite classes. In either case, we have a graph peninsula by taking $A$ to be an independent set of size $\frac n2$ and $B=\emptyset$. From this we see that a graph peninsula itself is not an obstruction, as both $H$ and $G$ have a fractional perfect matching and even a Hamilton cycle. However, these fractional perfect matchings/Hamilton cycles are extremely constrained, an issue which will appear fully in the graphon setting. Let $U$ and $W$ be graphons corresponding to $H$ and $G$ (see Figure~\ref{fig:bipgraphons}), which both have a peninsula. Due to random fluctuation, $\G(n,U)$ is asymptotically almost surely a complete slightly unbalanced bipartite graph, which does contain a narrow graph peninsula, and hence does not have a perfect fractional matching. In the random graph $\G(n,W)$, the number of vertices sampled from the constant-0 part (denoted as $I$ in Figure~\ref{fig:bipgraphons}) is with probability $\frac12+o(1)$ strictly bigger than $\frac{n}{2}$. Since these vertices form an independent set in $\G(n,W)$, they can again serve as the set $A$ of a narrow graph peninsula.

Note that the correspondence~\ref{corr2} also works in the converse direction. Indeed, if $G$ contains no narrow graph peninsula, then by the nontrivial direction of LP duality, together with the half-integrality of the vertex cover polytope, it follows that $G$ admits a perfect fractional matching. Thus, the absence of narrow graph peninsulae is equivalent to the existence of a perfect fractional matching. This dual viewpoint indicates that the notions of graph and graphon peninsula provide a natural structural language for formulating if-and-only-if criteria for perfect fractional matchings and, once appropriate additional conditions are imposed, for Hamiltonicity as well.

\subsection{Main result}
The main theorem is as follows.

\begin{thm}\label{thm:mainHC}
Suppose that $W:\Omega^2\to[0,1]$ is a graphon. 
Then $\G(n,W)$ is a.a.s.\ Hamiltonian if the following three conditions are fulfilled:
\begin{enumerate}[label=(\Roman*)]
    \item\label{en:mainHC1} $W$ is a connected graphon,
    \item\label{en:mainHC2} we have $\lim_{\alpha\searrow 0}\frac{\mu(\SMALLDEG_W(\alpha))}\alpha=0$, and
    \item\label{en:mainHC3} $W$ does not have a peninsula.
\end{enumerate}
Further, these conditions are necessary:
\begin{enumerate}[label=(\Alph*)]
    \item\label{en:mainHCnegative1} If $W$ is not a connected graphon, then $\lim_n\Probability[\G(n,W) \text{ is connected}]=0$.
    \item\label{en:mainHCnegative2} If $\liminf_{\alpha\searrow 0}\frac{\mu(\SMALLDEG_W(\alpha))}\alpha>0$, then $\liminf_n\Probability[\G(n,W) \text{ contains an isolated vertex}]>0$.
    \item\label{en:mainHCnegative3} If $W$ has a peninsula, then $\limsup_n\Probability[\G(n,W) \text{ has a matching on at least $n-t$ vertices}]\le \frac12$ for every $t=t(n)$ with $t=o(\sqrt n)$. 
\end{enumerate}
\end{thm}

Below, we treat the easy part of Theorem~\ref{thm:mainHC}, that is, that each of the conditions~\ref{en:mainHC1}, \ref{en:mainHC2}, and \ref{en:mainHC3} is necessary; for~\ref{en:mainHC2} we only recall the corresponding result of~\cite{ConnectivityPaper}. Then, in Section~\ref{sssec:outline}, we give an outline of the proof of sufficiency.

\subsubsection{Necessity of condition~\ref{en:mainHC1}}\label{ssec:Negative1}
Suppose $W$ is not connected, that is, we have $\Omega=S\sqcup T$ with $\mu(S), \mu(T) > 0$ and $\int_{S\times T}W= 0$. After the vertex-generating step~\ref{G1}, asymptotically almost surely, we have that the set $V_S\subseteq V$ of vertices whose type is in $S$ is nonempty, and that the set $V_T\subseteq V$ of vertices whose type is in $T$ is nonempty, too. We see that in the edge-generating step~\ref{G2}, almost surely, no edge is inserted between $V_S$ and $V_T$. Hence, $\G(n,W)$ is asymptotically almost surely a disconnected graph. 

\subsubsection{Necessity of condition~\ref{en:mainHC2}}\label{ssec:Negative2}
This is shown in Theorem~3(c) in~\cite{ConnectivityPaper}. The idea is that if $\liminf_{\alpha\searrow 0}\frac{\mu(\SMALLDEG_W(\alpha))}\alpha>0$, then for each $\eps>0$ there is a constant $c=c(\eps)>0$ such that with probability at least $c$, step~\ref{G1} produces an index $i\in[n]$ whose type satisfies $\deg_W(X_i)\le\frac{\eps}{n}$. Since the expected degree of such a vertex $i$ in $\G(n,W)$ is less than $\eps$, it ends up as an isolated vertex with probability greater than $1-\eps$.

\subsubsection{Necessity of condition~\ref{en:mainHC3}}\label{ssec:Negative3}
Let $A,B\subseteq \Omega$ be as in Definition~\ref{defi:peninsula}\ref{en:defiPen1}. Let $C:=\Omega\setminus(A\cup B)$. We have $\mu(C)= a$. Let $N_A$, $N_B$, and $N_C$ be the random variables counting the number of vertices which in~\ref{G1} get assigned types in $A$, in $B$, and in $C$, respectively. The crucial insight is that in~\ref{G2}, almost surely, no edges will be inserted in $\G(n,W)$ between a pair of vertices whose types were both in~$A$, or for which one type was in~$A$ and the other in~$B$. This is because $W$ is constant-0 on $A\times(A\cup B)$. This means that a function $f:[n]\to [0,1]$ which assigns to every vertex of type in $A$ value 0, to every vertex of type in $B$ value~$\frac12$, and to every vertex of type in $C$ value~1 is almost surely a fractional vertex cover of $\G(n,W)$. We have that $\|f\|_1= \frac{N_B}2+N_C=\frac12(n+N_C-N_A)$. Thus, if $N_A>N_C+t$, we found a fractional vertex cover of weight less than $\frac12(n-t)$, and so there exists no matching (or even fractional matching) of weight at least $\frac12(n-t)$. It remains to show that the event $\cE=\{N_A> N_C+t\}$ occurs with probability $\frac12-o(1)$.

Since $\mu(A)=\mu(C)=a$, the pair $(N_A,N_C)$ has the same distribution as $(N_C,N_A)$. As the three events $\{N_A>N_C+t\}$, $\{N_C>N_A+t\}$ and $\{|N_A-N_C|\le t\}$ form a partition of the probability space, we get
\[
\Probability[\cE]=\tfrac12\bigl(1-\Probability[|N_A-N_C|\le t]\bigr)\;.
\]
So, it only remains to show that $\Probability[|N_A-N_C|\le t]=o(1)$. To this end, write $N_A-N_C=\sum_{i\in[n]}Z_i$, where $Z_i:=\mathds{1}_{X_i\in A}-\mathds{1}_{X_i\in C}$. The random variables $Z_1,\ldots,Z_n$ are independent and identically distributed, with $\Expectation[Z_i]=0$ and $\Variance[Z_i]=2a$. Hence, by the Central Limit Theorem, $\frac1{\sqrt{n}}(N_A-N_C)$ converges in distribution to the centered normal distribution with variance $2a$, which is nondegenerate as $a>0$. Since $\frac{t}{\sqrt{n}}\to 0$, this gives $\Probability[|N_A-N_C|\le t]=o(1)$, as was needed to show.

\subsubsection{Outline of the proof}\label{sssec:outline}
In Section~\ref{sec:prel}, we recall the necessary background regarding combinatorial optimization (notions of fractional matchings and fractional vertex covers), graphons, probability theory, functional analysis, and Szemerédi's regularity lemma. The rest of the paper is devoted to the proof of Theorem~\ref{thm:mainHC}. From the assumptions of Theorem~\ref{thm:mainHC}, we establish that for $G=\G(n,W)$ and its regularized cluster graph $R$ (in the sense of the regularity lemma), the following properties hold with high probability:
\begin{enumerate}[label=(H\arabic*)]
    \item\label{en:outline1} $R$ contains an almost spanning connected component,
    \item\label{en:outline2} $G$ has very few vertices of low degree, and
    \item\label{en:outline3} $R$, after minor cleaning, is `robustly free of a graph peninsula'.
\end{enumerate}

When constructing a Hamilton cycle in $G$ using the regularity method, low-degree vertices pose an immediate obstacle. Therefore, in Section~\ref{sec:incorporatinglowvertices}, we establish property~\ref{en:outline2} and, at the same time, construct a small path system $\cP$ that covers all such vertices while ensuring that its end-vertices themselves are of high degree. This allows us to seamlessly append $\cP$ to the main bulk of our cycle later in the proof.
To construct the bulk of the cycle, we rely on a fractional variant of `Łuczak's connected matchings method'. (Arguably, the method was first used in~\cite{MR1676887} and the fractional variant in~\cite{MR3269902}, though in both cases there were almost concurrent and independent developments.) The goal is to find a near-perfect fractional matching in $R$, subdivide the regular pairs in its support according to the matching weights, and fill them with almost-spanning paths in $G$. Connectedness~\ref{en:outline1} then allows us to stitch these paths---along with the low-degree path system $\cP$---into a single cycle.

However, a notorious bottleneck in this method arises if the support of the fractional matching is bipartite. In such a setting, even microscopic fluctuations in the cluster sizes or interference from $\cP$ make it impossible to perfectly balance the bipartite classes, preventing the paths from forming a fully spanning cycle. 
It is exactly here that the structural properties from Section~\ref{ssec:matchingresults} become vital. In Section~\ref{StructureOfTheClusterGraph}, we formalize properties~\ref{en:outline1} and~\ref{en:outline3} via Proposition~\ref{prop:clustergraph}, showing that $R$ contains an almost spanning subgraph that is connected and `uniquely half-covered' (equivalently, free of graph peninsulae). As derived in Lemma~\ref{lem:OurCombOpt}, the latter property rescues the embedding on two fronts:
\begin{itemize}
    \item It guarantees the existence of the almost perfect fractional matching required to build the bulk of the cycle using Łuczak's connected matchings method.
    \item It guarantees that this subgraph is \emph{not} bipartite. 
\end{itemize}
Although in the actual proof we utilize each of these properties in relation to a different regularization of $G$, it suffices to think of just one regularization $R$ in this simplified overview.
Then this guaranteed non-bipartiteness provides the necessary odd cycles to implement the absorption method, systematically introduced by Rödl, Ruciński, and Szemerédi~\cite{RoRuSzAbsor06} in~2006 and used in many extremal results since. In Section~\ref{sec:Absoption}, we construct an absorber suitable for our purposes. Because we are not trapped in a rigid bipartite parity, this absorber can `swallow' the small, unstructured set of leftover vertices that Łuczak's connected matchings method could not cover. In Section~\ref{sec:proof}, we put all these components together and construct the desired Hamilton cycle. Finally, in Section~\ref{sec:Concluding} we refine the negative direction of Theorem~\ref{thm:mainHC} by identifying graphons for which $\G(n,W)$ is a.a.s.\ \emph{not} Hamiltonian, and we discuss sparser inhomogeneous random graphs and other spanning structures in $\G(n,W)$.

\section{Preliminaries}\label{sec:prel}
\subsection{Graphs}
In this section we recall two easy graph-theoretic results. Further, deeper graph theory preliminaries concern combinatorial optimization and Szemerédi's regularity lemma. These are given in separate Sections~\ref{ssec:matchingresults} and~\ref{ssec:RegLemma}, respectively.

We use standard graph-theoretic notation. Given a graph $G$, we write $v(G)$ for the number of its vertices (also called the \emph{order} of $G$) and $e(G)$ for the number of its edges. Given a vertex $v$ of $G$, we write $N_G(v)$ for its neighbourhood and $\deg_G(v):=|N_G(v)|$ for its degree. Given two disjoint sets $X,Y\subseteq V(G)$, we write $e_G(X,Y)$ for the number of edges of $G$ with one end-vertex in $X$ and the other one in $Y$. We drop the subscript $G$ whenever the graph is clear from the context.

We begin with the following lemma that provides a walk of odd length in any connected non-bipartite graph, where the \emph{length} of a walk is its number of edges.

\begin{lem}\label{lem:closedwalks}
Let $H$ be a connected non-bipartite graph on $r$ vertices. Then for every pair of vertices $i,j \in V(H)$ (not necessarily distinct), there exists a walk $Q_{i,j}$ of odd length at most $2r-1$ starting at $i$ and ending at $j$.
\end{lem}
\begin{proof}
Let $T$ be a spanning tree of $H$, and let us 2-colour $V(H)$ according to the parity of the distance from a fixed vertex in $T$. Since $H$ is non-bipartite, there is an edge $e \in E(H) \setminus E(T)$ whose two end-vertices have the same colour; such an edge creates an odd cycle $C$ when added to $T$.

Consider the subgraph $H' = T \cup \{e\}$. In $H'$, there is a unique path $P_{i,j}$ between $i$ and $j$ utilizing only edges of $T$, and an alternative walk $P'_{i,j}$ that connects $i$ and $j$ by using the odd cycle $C$ via the edge $e$. Since $C$ has odd length, the length of $P_{i,j}$ and the length of $P'_{i,j}$ necessarily have opposite parities. Thus, one of them must be an odd walk.

The length of the tree path $P_{i,j}$ is trivially bounded by $r-1$. The walk $P'_{i,j}$ traverses the edges of $T \setminus E(C)$ at most twice, and the edges of $C$ at most once. Hence its length can also be bounded from above by $2r-1$.
\end{proof}

By a \emph{binary tree} we mean a tree in which there is exactly one vertex of degree~2 (the \emph{root}), any number of vertices of degree~3 (\emph{internal vertices}) and any number of vertices of degree~1 (\emph{leaves}).
\begin{fact}\label{fact:decomposetree}
Suppose that $T$ is a binary tree. Then there exists a system $\mathcal{F}$ of vertex-disjoint paths $\{P_i\subseteq T\}_i$ such that each $P_i$ has at least 3 vertices, both end-vertices of each $P_i$ are leaves of $T$, and every vertex of $T$ which is not a leaf lies on one of the paths $P_i$.
\end{fact}
\begin{proof}
We proceed by induction on $v(T)$. Let us view $T$ as rooted at its root $\rho$; then each vertex of $T$ has either two children (if it is the root or an internal vertex) or none (if it is a leaf). Starting at each of the two children of $\rho$ and descending to an arbitrary child as long as possible, we obtain a path $P\subseteq T$ which passes through $\rho$, has at least 3 vertices, and both of whose end-vertices are leaves of $T$. Each component of $T-V(P)$ is a subtree of $T$ rooted at a child of a vertex of $P$, and thus is either a single leaf of $T$, or a binary tree all of whose leaves are leaves of $T$. Taking $P$ together with the path systems provided by the induction hypothesis for the latter components, we get the desired system $\mathcal{F}$.
\end{proof}
Let us stress that $\mathcal{F}$ need not cover all the leaves of $T$; indeed, this cannot be required in general, as a parity argument shows.

\subsection{Fractional matchings and fractional vertex covers in graphs}\label{ssec:matchingresults}

As motivated in Section~\ref{sssec:peninsula}, the notion of fractional vertex covers plays a critical role in our proof. The underlying combinatorial optimization principles we recall below are classical; what matters for us is that translating the absence of graph peninsulae into the language of fractional vertex covers provides the structural engine for our main argument. It allows us to simultaneously deduce two rather different macroscopic properties required for our embedding scheme: non-bipartiteness (necessary for constructing absorbing structures) and the existence of fractional perfect matchings (necessary for embedding the bulk of the Hamilton cycle). 

We begin by recalling the standard terminology. Suppose that $G=(V,E)$ is a graph. A function $m:E\to [0,1]$ is a \emph{fractional matching in $G$} if for every $v\in V$ we have $\sum_{w\in N_G(v)} m(vw)\le 1$. The \emph{total weight} of $m$ is defined as $\sum_{e\in E}m(e)$. A fractional matching is \emph{perfect} if its total weight equals $\frac{|V|}{2}$. This is equivalent to having $\sum_{w\in N_G(v)} m(vw)=1$ for every $v\in V$. The \emph{fractional matching number of $G$}, denoted $\mathrm{fmn}(G)$, is defined as the maximum of total weights over all fractional matchings in $G$. 

Dually, a function $f:V\to [0,1]$ is a \emph{fractional vertex cover in $G$} if for every $vw\in E$ we have $f(v)+f(w)\ge 1$. The \emph{total weight} of $f$ is defined as $\sum_{v\in V}f(v)$. The \emph{fractional vertex cover number of $G$}, denoted $\mathrm{fvcn}(G)$, is defined as the minimum of total weights over all fractional vertex covers in $G$.

The structural properties we rely on stem from three fundamental facts about fractional matchings and fractional vertex covers. The first is arguably the most important instance of LP duality. The second and third facts follow from the half-integrality of the fractional matching polytope and of the fractional vertex cover polytope (see, e.g., \cite[Sections~30.2--30.3]{SchrijverCombOpt}). Recall that a function (whether a fractional matching or a fractional vertex cover) is \emph{half-integral} if its range is a subset of $\{0,\frac12,1\}$.

\begin{fact}\label{fact:CombOpt}
Suppose that $G$ is a graph.
\begin{enumerate}[label=(\roman*)]
    \item\label{en:CombOptLP} We have $\mathrm{fmn}(G)=\mathrm{fvcn}(G)$.
    \item\label{en:CombOptHalfMatching} There exists a half-integral fractional matching of total weight $\mathrm{fmn}(G)$.
    \item\label{en:CombOptHalfVertexCover} There exists a half-integral fractional vertex cover of total weight $\mathrm{fvcn}(G)$.
\end{enumerate}
\end{fact}

We say that a graph $G$ is \emph{uniquely half-covered} if the only half-integral vertex cover of total weight at most $\frac{v(G)}{2}$ is the constant-$\frac12$ function. The critical nature of this property becomes obvious when we recall~\ref{corr1}: a graph has a graph peninsula if and only if it is \emph{not} uniquely half-covered. 

The following lemma distills the power of this property into the exact two conditions we will need for our main proof. While its proof is a straightforward consequence of the classical theory, this lemma is a key ingredient of our embedding strategy.

\begin{lem}\label{lem:OurCombOpt}
Suppose that $G$ is a uniquely half-covered graph. Then:
\begin{enumerate}[label=(\roman*)]
    \item\label{en:OurCombOptNonBip} The graph $G$ is not bipartite.
    \item\label{en:OurCombOptPerfect} There exists a half-integral perfect fractional matching in $G$.
\end{enumerate}
\end{lem}
\begin{proof}
 We prove part~\ref{en:OurCombOptNonBip} by contrapositive. Suppose that $G$ is bipartite, and let $V(G)=A\sqcup B$ be a bipartition with $|A|\le |B|$. Then the function $f:V(G)\to\{0,1\}$ which is the indicator function of the set $A$ is a half-integral vertex cover of total weight $|A|\le \frac{v(G)}{2}$. As $f$ attains only the values 0 and~1, it differs from the constant-$\frac12$ function, and we conclude that $G$ is not uniquely half-covered.

 For part~\ref{en:OurCombOptPerfect}, assume that $G$ is uniquely half-covered. Using Fact~\ref{fact:CombOpt}\ref{en:CombOptHalfVertexCover}, $G$ does not have any fractional vertex cover of weight less than $\frac{v(G)}2$. We get that $\mathrm{fvcn}(G)=\frac{v(G)}{2}$. The assertion then follows immediately from Fact~\ref{fact:CombOpt}\ref{en:CombOptLP} and Fact~\ref{fact:CombOpt}\ref{en:CombOptHalfMatching}.
\end{proof}
\subsection{Graphons}\label{ssec:graphonpreliminaries}
We need some basics of the theory of graphons. We follow the excellent monograph~\cite{MR3012035}. As mentioned earlier, we work with an arbitrary atomless standard measure space $\Omega$ with probability measure $\mu$ and an implicit sigma-algebra; recall also our convention that all the sets and all the functions we consider are measurable. In the theory of graphons, one often works with $\Omega=[0,1]$ equipped with the Lebesgue measure. This more concrete setting is equivalent to ours by the Lebesgue isomorphism theorem. A \emph{signed graphon} is any symmetric function $W\in L^\infty(\Omega^2)$. Further, we say that $W$ is a \emph{graphon} if $W(x,y)\in[0,1]$ for $\mu^2$-almost every $(x,y)\in\Omega^2$. If $W$ is a signed graphon, then its \emph{cut norm} is defined as $\cutn{W}:=\sup_{S,T\subseteq \Omega}\left|\int_{S\times T}W\right|$. In particular, we can introduce the \emph{cut norm distance} between two graphons $W_1$ and $W_2$ as $\cutn{W_1-W_2}$. 

A \emph{weighted graph} is a graph $G=(V,E)$ together with a weight function $w:E\to[0,1]$. For a vertex $v$ of a weighted graph $(G,w)$ we write $\wdeg_G(v):=\sum_{u\in N_G(v)}w(uv)$ for its \emph{weighted degree}, and we keep the notation $\deg_G(v)=|N_G(v)|$ for the number of its neighbours. Since all the weights are at most~1, we always have
\begin{equation}\label{eq:wdegdeg}
\wdeg_G(v)\le \deg_G(v)\;.
\end{equation}
Whenever a notion that we defined for graphs is applied to a weighted graph---this concerns, for instance, fractional matchings, fractional vertex covers, connectivity, and the property of being uniquely half-covered---it is meant for the underlying graph $G$, that is, the weights are ignored.
Given $\delta>0$, we write
\[
\SMALLDEG_{G}(\delta):=\{v\in V(G)\colon \wdeg_G(v)\le\delta v(G)\}
\]
for the set of vertices of small weighted degree; the normalization by $v(G)$ makes this consistent with the set $\SMALLDEG_W(\delta)$ introduced for graphons in condition~\ref{C2}, as~\eqref{eq:wdegrepr} below shows.
Given a weighted graph $(G,w)$, we construct a \emph{graphon representation} $U$ as follows. We partition $\Omega=\bigsqcup_{v\in V}\Omega_v$ into sets of measure $\frac{1}{v(G)}$ each. For a pair $(u,v)\in V^2$, we define $U_{\restriction \Omega_u\times \Omega_v}$ to be constant-0 if $uv$ does not form an edge of $G$. If it does, we define $U_{\restriction \Omega_u\times \Omega_v}$ to be constant-$w(uv)$. Note that $U$ is not uniquely defined, as it depends on the partition $\Omega=\bigsqcup_{v\in V}\Omega_v$. Note also that the graphon representation records the \emph{weighted} degrees of $G$: for every $v\in V$ and every $x\in \Omega_v$ we have
\begin{equation}\label{eq:wdegrepr}
\deg_U(x)=\sum_{u\in N_G(v)}\mu(\Omega_u)w(uv)=\frac{\wdeg_G(v)}{v(G)}\;,
\end{equation}
and consequently $\SMALLDEG_U(\delta)=\bigcup_{v\in \SMALLDEG_G(\delta)}\Omega_v$ for every $\delta>0$. We say that a \emph{graphon $W$ is at the cut distance at most $\delta$ from a weighted graph $(G,w)$} if there exists a graphon representation $U$ as above with $\cutn{W-U}\le \delta$.

The following lemma can be deduced from well-known results in the theory of graphons. We include a self-contained proof.
\begin{lem}\label{lem:delete}
Suppose that $\delta\in[0,1)$, $R$ is a weighted graph of order $n$, and $W_R$ is its graphon representation. Suppose that a weighted graph $R^*$ is obtained from $R$ by deleting at most $\delta n$ vertices. Then $R^*$ has a graphon representation $W_{R^*}$ such that $\cutn{W_R-W_{R^*}}\le 2\delta$.
\end{lem}
\begin{proof}
Suppose that $\Omega=\bigsqcup_{v\in V(R)}\Omega_v$ is a partition of $\Omega$ which was used to define $W_R$. Fix a partition $\Omega=\bigsqcup_{v\in V(R^*)}\Omega^*_v$ such that for every $v\in V(R^*)$ we have $\Omega_v\subseteq \Omega_v^*$. Take $W_{R^*}$ to be the representation of $R^*$ with respect to this partition. Crucially, observe that for $u,v\in V(R^*)$ and every $(x,y)\in \Omega_u\times \Omega_v$, we have $W_R(x,y)=W_{R^*}(x,y)$. Thus,
\[
\cutn{W_R-W_{R^*}}\le \left\| W_R-W_{R^*} \right\|_1 \le
\mu^2\left(\Omega^2\setminus\left(\cup_{v\in V(R^*)}\Omega_v\right)^2\right)
=1-\left(\frac{v(R^*)}{v(R)}\right)^2\le 2\delta\;,
\]
as was needed.
\end{proof}

Furthermore, we need two observations about sampling a graph from a graphon. Firstly, the graph degrees are proportional to the graphon degrees. This well-known fact follows from a direct application of the Chernoff bound.

\begin{fact}\label{fact:sampledeg}
    Let $W$ be a graphon and $\eps>0$. There exists $n_0\in\N$ such that for every $n\ge n_0$ and for $G\sim \G(n,W)$ the following holds. The probability that for each $i\in[n]$ we have $\deg_G(i)=(\deg_W(X_i)\pm \eps) n$ is at least $1-\eps$. (Here $X_i$ is the type of the vertex $i\in [n]$ in~\ref{G1}.)
\end{fact}

Secondly, a sample is close to the graphon in cut norm distance. This follows from the so-called second sampling lemma~\cite[Lemma 10.16]{MR3012035}.

\begin{lem}\label{lem:sampledist}
    Let $W$ be a graphon and $\eps>0$. Then asymptotically almost surely there exists a graphon representation $W_G$ of $G\sim \G(n,W)$ such that $\cutn{W_G-W}\le\eps$.
\end{lem}

We will also make use of the following result which states that a weighted graph which is sufficiently close to a connected graphon cannot have two linear-sized connected components. It is a reformulation of Theorem~1.10(i) in~\cite{ConnectednessGraphons}, which is stated there for a sequence of unweighted graphs converging to a connected graph limit; we include the short derivation.

\begin{prop}\label{prop:connectedgraphon}
Suppose that $W$ is a connected graphon. For every $\zeta>0$ there exists $\kappa>0$ with the following property. Suppose that $H$ is a weighted graph whose cut distance from $W$ is at most $\kappa$. Then for every set $X\subseteq V(H)$ with $|X|\in[\zeta v(H),\frac12 v(H)]$ we have $e_H(X,V(H)\setminus X)>0$.
\end{prop}
\begin{proof}
We may assume that $\zeta\le\frac12$, as otherwise the assertion is vacuous. Suppose that it fails for some such $\zeta$. Then for every $m\in\N$ there are a weighted graph $H_m$ with a graphon representation $W_{H_m}$ satisfying $\cutn{W-W_{H_m}}\le\frac1m$, and a set $X_m\subseteq V(H_m)$ with $|X_m|\in[\zeta v(H_m),\frac12 v(H_m)]$ such that no edge of $H_m$ joins $X_m$ to $V(H_m)\setminus X_m$. Let $\Omega_{X_m}\subseteq \Omega$ be the union of those parts of the partition underlying $W_{H_m}$ which correspond to the vertices of $X_m$. Then $\mu(\Omega_{X_m})=\frac{|X_m|}{v(H_m)}\in[\zeta,\frac12]$, and $W_{H_m}$ is constant-0 on $\Omega_{X_m}\times(\Omega\setminus\Omega_{X_m})$.

Fix $m$ and consider $G\sim\G(N,W_{H_m})$, and let $Y\subseteq V(G)$ be the set of vertices whose type lies in $\Omega_{X_m}$. Almost surely no edge of $G$ joins $Y$ to $V(G)\setminus Y$, since $W_{H_m}$ vanishes on $\Omega_{X_m}\times(\Omega\setminus\Omega_{X_m})$. Further, $|Y|=\bigl(\mu(\Omega_{X_m})\pm\frac\zeta4\bigr)N$ with probability $1-o(1)$ by the law of large numbers, and by Lemma~\ref{lem:sampledist} there is, again with probability $1-o(1)$, a graphon representation $W_G$ of $G$ with $\cutn{W_G-W_{H_m}}\le\frac1m$. Hence for every $m$ we may fix $N_m\ge m$, a graph $G_m$ on $N_m$ vertices and a set $Y_m\subseteq V(G_m)$ such that
\[
\cutn{W_{G_m}-W}\le\tfrac2m,\qquad
|Y_m|,\,|V(G_m)\setminus Y_m|\ge\tfrac\zeta4 N_m,\qquad
e_{G_m}\bigl(Y_m,V(G_m)\setminus Y_m\bigr)=0\;.
\]
Since $\cutn{W_{G_m}-W}\to 0$, the sequence $G_1,G_2,\ldots$ converges to the graph limit represented by $W$ (see, e.g.,~\cite[Chapter~11]{MR3012035}), and that graph limit is connected because $W$ is, by Theorem~1.16 in~\cite{ConnectednessGraphons}. As $v(G_m)=N_m\to\infty$, Theorem~1.10(i) in~\cite{ConnectednessGraphons} applied with $\delta:=\frac\zeta4$ provides an $\eps>0$ such that $e_{G_m}(Y_m,V(G_m)\setminus Y_m)\ge \eps N_m^2$ for all large $m$, a contradiction.
\end{proof}

\subsection{Probability theory}
We need two versions of the Chernoff bound which can for example be derived from~\cite[Theorem 2.1]{JLR}.
\begin{lem}\label{lem:Chernoff}
Suppose that $n\in\N$, $d\in[0,1]$ and $C\ge 1$. Suppose that $B$ is a random variable with binomial distribution with parameters $n$ and $d$.
\begin{enumerate}[(i)]
    \item\label{it:Chernoff1}$\Probability[B\le \frac12nd]\le \exp(-nd/8)$, and
    \item\label{it:Chernoff2}$\Probability[B\ge Cnd]\le \exp(-ndC\ln(C/e))$.
\end{enumerate}
\end{lem}

\subsection{Functional analysis}\label{ssec:FunctionalAnalysis}

In this section, we collect several definitions and tools from functional analysis tailored to our setting. Throughout, let $(\Omega, \mu)$ be a probability space with an implicit sigma-algebra. For a measurable function $f: \Omega \to \R$, we define its \emph{support} as $\supp(f) = \{x \in \Omega : f(x) \neq 0\}$. For any $\delta > 0$, we define the \emph{$\delta$-support} as $\supp_\delta(f) = \{x \in \Omega : f(x) > \delta\}$.

The space $L^\infty(\Omega)$ is the dual of $L^1(\Omega)$. This duality induces the \emph{weak* topology}. While the underlying machinery of Banach spaces is quite deep, the discrete-minded reader can safely view this topology as a convenient formalism for the ``convergence of averages'' over arbitrary test sets (for a standard reference, see, e.g., \cite[Chapter 5]{Folland1999}). For our purposes, we use the following concrete characterization of weak* convergence of uniformly bounded sequences, which is the only setting we shall need: a uniformly bounded sequence of functions $g_1, g_2, \ldots \in L^\infty(\Omega)$ \emph{converges weak*} to a function $g \in L^\infty(\Omega)$ if, for every set $X \subseteq \Omega$, we have 
\[ \lim_{n \to \infty} \int_X g_n \diff\mu = \int_X g \diff\mu .\]

The following fact is trivial to verify. We record it here for a later reference.
\begin{fact}\label{fact:weakstarbounded}
 Suppose that a sequence of functions $g_1, g_2, \ldots :\Omega\to[0,1]$ converges weak* to a function $g \in L^\infty(\Omega)$. Then we have $g(x)\in[0,1]$ for almost every $x\in\Omega$.
\end{fact}

A central pillar of our analysis is the \emph{Banach--Alaoglu Theorem}. In the current context, it ensures that the unit ball of $L^\infty(\Omega)$ is weak* compact. Since $\Omega$ is a standard probability space, the space $L^1(\Omega)$ is separable, and therefore this unit ball is even weak* sequentially compact. Specifically, given any sequence of functions $g_1, g_2, \ldots : \Omega \to [0,1]$, there exists a subsequence $g_{n_1}, g_{n_2}, \ldots$ and a function $g\in L^\infty(\Omega)$ such that $g_{n_k}$ converges weak* to $g$ as $k \to \infty$. This compactness-based framework---which allows for the extraction of a continuous limit from a sequence of discrete objects (such as vertex covers of growing finite graphs)---has proven highly effective in the study of graphons, as demonstrated in~\cite{HLADKY2020103108, DolezalHladky:MatchingPolytons, MR3879960}.

\subsection{Asymptotic notation}
When using multiple parameters, we will sometimes define them via \emph{hierarchies}.
Those are defined from right to left.
That is, we denote by $\beta\ll \alpha$ that $\beta$ is chosen sufficiently small with respect to $\alpha$. More precisely, for any $\alpha>0$ there exists $\beta_0>0$ such that for any $0<\beta\le \beta_0$ the subsequent statement holds.
This extends in the natural way to hierarchies consisting of several parameters.

Further, we write $x=y\pm z$ as a shorthand for $x\in[y-z,y+z]$. Given a sequence $\cE_1,\cE_2,\ldots$ of events, we say that $\cE_n$ occurs \emph{asymptotically almost surely} (\emph{a.a.s.}), or \emph{with high probability}, if $\Probability[\cE_n]\to 1$ as $n\to\infty$.

\subsection{Regularity lemma}\label{ssec:RegLemma}
A substantial part of our proof relies on the regularity method for which we give the necessary background in this section. For two disjoint nonempty sets $X$ and $Y$ of vertices of a graph $G$, the \emph{density of $(X,Y)$} is defined as $d_G(X,Y)=\frac{e_G(X,Y)}{|X||Y|}$. A pair $(A,B)$ of disjoint nonempty sets of vertices is \emph{$\eps$-regular} if $\left|d_G(A,B)-d_{G}(A',B')\right|\le \eps$ for every $A'\subseteq A$ and $B'\subseteq B$ with $|A'|\ge\eps|A|$ and $|B'|\ge\eps|B|$. Subsets of a regular pair inherit some regularity via the following \emph{slicing lemma}.

\begin{fact}[{\cite[Fact 1.5]{regularity-oldsurvey}}]\label{fact:regpairinherit}
Let $G$ be a graph and $\alpha>\eps>0$. If
$(X,Y)$ is an $\eps$-regular pair of density $d$ in $G$ and $X'\subseteq X$ with $|X'|\ge\alpha|X|$ and $Y'\subseteq Y$ with $|Y'|\ge\alpha|Y|$, then $(X',Y')$ is an 
$\eps'$-regular pair with $\eps'=\max\{\eps/\alpha,2\eps\}$ and for its density $d'$ we have $|d'-d|<\eps$.
\end{fact}

We say that a partition $V(G)=V_1\sqcup \ldots \sqcup V_\ell$, where the sets $V_1,\ldots,V_\ell$ are called \emph{clusters}, is an \emph{$\eps$-regular partition} of $G$ if we have
\begin{enumerate}[label=(R\arabic*)]
    \item \label{R:1} $\big| |V_i|-|V_j|\big|\le \eps |V_i|$ for every $i,j\in[\ell]$, and
	\item\label{R:2} for all but at most $\eps\ell^2$ many unordered pairs $\{i,j\}$ of distinct indices $i,j\in[\ell]$ we have that $(V_i,V_j)$ is an $\eps$-regular pair.
\end{enumerate}

For constructing the Hamilton cycle our proof is based on the well-known \emph{regularity lemma} in its \emph{degree form}.
The theorem below is a modification of it in two respects: the partition is required to refine a given prepartition $V(G)=X\sqcup Y$, and there is no exceptional cluster. The price of the latter is that the degree condition survives only outside a small exceptional set $Z$. We deduce the theorem from the usual degree form below; two remarks explaining why the statement has to take this shape come first.

\begin{thm}[{cf.~\cite[Theorem 1.10]{regularity-oldsurvey}}]\label{thm:reglem}
    For every $\eps>0$, there exist constants $n_0$ and $T$ such that for every $n$-vertex graph $G$ with~$n\ge n_0$ and with~$V(G)=X\sqcup Y$, where $|X|,|Y|\in\{0\}\cup[\eps n,n]$, and any $d\in[2\eps,1]$ the following holds.
    There is an $\eps$-regular partition~$V_1,\dots, V_t$ with $\frac1\eps< t\le T$ and such that~$V_i\subseteq X$ or~$V_i\subseteq Y$ for every~$i\in [t]$.
    Moreover, if we let~$G'$ be the graph after removing all edges~$uv\in E(G)$ with
    \begin{itemize}
        \item $u,v\in V_i$ for some~$i\in [t]$, 
        \item $u\in V_i$,~$v\in V_j$, and~$(V_i,V_j)$ is not~$\eps$-regular, or
        \item $u\in V_i$,~$v\in V_j$, and~$(V_i,V_j)$ has density smaller than~$d$, 
    \end{itemize}
    then there is a set~$Z\subseteq V(G)$ with~$|Z|\le \eps n$ such that $\deg_{G}(v) < \deg_{G'}(v) + 2dn$ for every vertex~$v\in V(G)\setminus Z$.
\end{thm}

Let us stress that in Theorem~\ref{thm:reglem} one cannot insist, as is otherwise customary, that the sizes of the clusters differ by at most~1. Indeed, suppose that $|V_i|\in\{m,m+1\}$ for every $i\in[t]$, where necessarily $m=\lfloor \frac{n}{t}\rfloor$, and that $X$ is the union of $k$ of the clusters. Then $|X|-km\in[0,k]$. Since $k\le t\le T$, this leaves at most $2T^3$ possible values of $|X|$, and thus for $n>4T^3$ the condition fails for the vast majority of the admissible values of $|X|$. This is the reason for the relative form of~\ref{R:1}.

Likewise, the degree condition cannot be required for every vertex. In the customary degree form the partition comes with an exceptional cluster $V_0$, and the edges incident to $V_0$ are kept in $G'$; this is exactly what makes the degree condition hold everywhere, and redistributing $V_0$ destroys it. For an example, take $\eps<\frac18$ and $d:=2\eps$, and let $G$ be the star $K_{1,n-1}$ with the trivial prepartition $X:=\emptyset$. Each pair of distinct clusters then has density at most $\frac{1}{\min_k |V_k|}$, which is smaller than $d$ once $n$ is large. Hence $G'$ is the empty graph---the edges inside a cluster are removed by the first bullet, and all the remaining ones by the third---while the centre of the star has degree $n-1>2dn$.

\begin{proof}[Proof of Theorem~\ref{thm:reglem}]
We may assume that $\eps\le\tfrac12$, since for $\eps>\tfrac12$ the interval $[2\eps,1]$ is empty and there is nothing to prove. We may further assume that $d\le\tfrac12$. Indeed, for $d>\tfrac12$ we have $2dn>n>\deg_G(v)$ for every vertex~$v$, so the last assertion of the theorem holds with $Z:=\emptyset$ whatever $G'$ is, while the remaining assertions do not involve~$d$.

Set
\begin{equation}\label{eq:reglemconst}
\eps':=\tfrac{\eps^3}{200}\;,\qquad \delta:=\tfrac{3\eps'}{\eps}=\tfrac{3\eps^2}{200}\;,\qquad d^\sharp:=d+2\delta\;,
\end{equation}
and note that $\eps'\le\tfrac\eps2$, that $\delta\le\eps$, that $d^\sharp\le\tfrac12+\tfrac{3}{400}<1$, and that
\begin{equation}\label{eq:regenlargeconst}
\eps'+3\sqrt{\delta}\le\tfrac{\eps}{200}+3\eps\sqrt{\tfrac{3}{200}}<\tfrac{2\eps}{5}\;.
\end{equation}

We use the regularity lemma in its degree form~\cite[Theorem 1.10]{regularity-oldsurvey}, in the variant in which the partition refines a given prepartition; the latter is obtained in the usual way, by starting the iterative refinement procedure in the proof with the partition $\{X,Y\}$. Applied to $G$ with the parameters $\eps'$ and $d^\sharp$, it provides constants $n_0$ and $T=T(\eps')$ and, for every $n\ge n_0$, a partition $V(G)=V_0\sqcup W_1\sqcup\ldots\sqcup W_t$ with the following properties.
\begin{enumerate}[label=(D\arabic*)]
    \item\label{D:sizes} We have $t\le T$ and $|V_0|\le\eps' n$, the sets $W_1,\ldots,W_t$ are all of the same size $m\le\lceil\eps' n\rceil$, and each of them is contained in $X$ or in $Y$. In particular, $t\ge\tfrac{(1-\eps')n}{\lceil \eps' n\rceil}\ge\tfrac{1}{2\eps'}$ after enlarging $n_0$ if necessary.
    \item\label{D:irreg} All but at most $\eps' t^2$ of the unordered pairs $\{i,j\}$ of distinct indices $i,j\in[t]$ are such that $(W_i,W_j)$ is $\eps'$-regular.
    \item\label{D:deg} Let $G^\sharp$ be the graph obtained from $G$ by removing all the edges with both end-vertices in the same set $W_i$, all the edges between pairs $(W_i,W_j)$ which are not $\eps'$-regular, and all the edges between pairs $(W_i,W_j)$ of density smaller than $d^\sharp$. Then $\deg_{G^\sharp}(v)>\deg_G(v)-(d^\sharp+\eps')n$ for every $v\in V(G)$.
\end{enumerate}
Let us stress that no edge incident to $V_0$ is removed in~\ref{D:deg}; this is precisely what makes the estimate there hold for \emph{every} vertex, and it is what we shall lose by dissolving~$V_0$.

We now redistribute $V_0$: each vertex of $V_0\cap X$ is sent to one of the sets $W_i$ contained in $X$, and each vertex of $V_0\cap Y$ to one of the sets $W_i$ contained in $Y$, as evenly as possible. Writing $V_i\supseteq W_i$ for the sets obtained in this way, we get a partition $V(G)=V_1\sqcup\ldots\sqcup V_t$ in which each $V_i$ is contained in $X$ or in $Y$. Whenever $X\neq\emptyset$ we have $|X|\ge\eps n$; since $|X|\le\eps' n+m\cdot\bigl|\{i\in[t]:W_i\subseteq X\}\bigr|$ and $m\le\tfrac{n}{t}$, at least $(\eps-\eps')t\ge\tfrac{\eps}{2}t$ of the sets $W_i$ are contained in $X$, and the same applies to $Y$. Hence each $W_i$ receives at most $\tfrac{2\eps' n}{\eps t}+1$ vertices, and as $m\ge\tfrac{(1-\eps')n}{t}$ we get, after enlarging $n_0$ if necessary,
\begin{equation}\label{eq:clustergrowth}
m\le|V_i|\le\Bigl(1+\tfrac{2\eps'}{\eps(1-\eps')}+\tfrac{t}{(1-\eps')n}\Bigr)m\le(1+\delta)m\qquad\textrm{for every }i\in[t]\;.
\end{equation}
Such a mild enlargement is harmless for regularity.
\begin{claim}\label{clm:enlargepair}
Let $(A,B)$ be an $\eps'$-regular pair, and let $A\subseteq A^+$ and $B\subseteq B^+$ be such that $|A^+|\le(1+\delta)|A|$ and $|B^+|\le(1+\delta)|B|$. Then $(A^+,B^+)$ is an $(\eps'+3\sqrt{\delta})$-regular pair, and $\bigl|d(A^+,B^+)-d(A,B)\bigr|\le2\delta$.
\end{claim}
\begin{claimproof}
We first record that for all nonempty sets $A''\subseteq A'\subseteq A^+$ and $B''\subseteq B'\subseteq B^+$ we have
\begin{equation}\label{eq:densityshrink}
\bigl|d(A',B')-d(A'',B'')\bigr|\le1-\tfrac{|A''||B''|}{|A'||B'|}\;.
\end{equation}
Indeed, write $\varrho:=\tfrac{|A''||B''|}{|A'||B'|}\le1$. The pairs counted by $e(A',B')$ but not by $e(A'',B'')$ lie in $(A'\times B')\setminus(A''\times B'')$, and so $e(A'',B'')\le e(A',B')\le e(A'',B'')+|A'||B'|-|A''||B''|$. Dividing by $|A'||B'|$ gives $\varrho\, d(A'',B'')\le d(A',B')\le\varrho\, d(A'',B'')+1-\varrho$, and~\eqref{eq:densityshrink} follows since $d(A'',B'')\le1$.

Applying~\eqref{eq:densityshrink} with $A'=A^+$, $A''=A$, $B'=B^+$ and $B''=B$ gives $\bigl|d(A^+,B^+)-d(A,B)\bigr|\le1-(1+\delta)^{-2}\le2\delta$, which is the second assertion of the claim.

For the first one, write $\eps'':=\eps'+3\sqrt{\delta}$ and let $A'\subseteq A^+$ and $B'\subseteq B^+$ satisfy $|A'|\ge\eps''|A^+|$ and $|B'|\ge\eps''|B^+|$. Set $A'':=A'\cap A$ and $B'':=B'\cap B$. From $|A'|-|A''|\le|A^+|-|A|\le\delta|A^+|$ we get $|A''|\ge(\eps''-\delta)|A^+|\ge\eps'|A|$, where we used $\delta\le3\sqrt{\delta}$ and $|A|\le|A^+|$; the same bound holds for $|B''|$. Hence $\bigl|d(A'',B'')-d(A,B)\bigr|\le\eps'$ by the $\eps'$-regularity of $(A,B)$. Furthermore $\tfrac{|A''|}{|A'|}\ge1-\tfrac{\delta|A^+|}{|A'|}\ge1-\tfrac{\delta}{\eps''}$, and likewise for $B$, so~\eqref{eq:densityshrink} gives
\[
\bigl|d(A',B')-d(A'',B'')\bigr|\le1-\Bigl(1-\tfrac{\delta}{\eps''}\Bigr)^2\le\tfrac{2\delta}{\eps''}\le\tfrac{2\delta}{3\sqrt{\delta}}\le\sqrt{\delta}\;.
\]
Altogether $\bigl|d(A',B')-d(A,B)\bigr|\le\eps'+\sqrt{\delta}\le\eps''$, as was needed.
\end{claimproof}

By~\eqref{eq:clustergrowth} and Claim~\ref{clm:enlargepair}, every $\eps'$-regular pair $(W_i,W_j)$ gives rise to a pair $(V_i,V_j)$ which is $(\eps'+3\sqrt{\delta})$-regular, hence $\eps$-regular by~\eqref{eq:regenlargeconst}, and which satisfies $d(V_i,V_j)\ge d(W_i,W_j)-2\delta$. Combined with~\ref{D:irreg}, this shows that all but at most $\eps' t^2\le\eps t^2$ of the pairs $(V_i,V_j)$ are $\eps$-regular, which is~\ref{R:2}. Moreover~\eqref{eq:clustergrowth} gives $\bigl||V_i|-|V_j|\bigr|\le\delta m\le\eps|V_i|$ for all $i,j\in[t]$, which is~\ref{R:1}, and $\tfrac1\eps<\tfrac1{2\eps'}\le t\le T$ by~\ref{D:sizes}. So $V_1,\ldots,V_t$ is an $\eps$-regular partition of the required kind.

It remains to verify the degree condition, and we set $Z:=V_0$, so that $|Z|\le\eps' n\le\eps n$ by~\ref{D:sizes}. We claim that every edge of $G^\sharp$ with no end-vertex in $V_0$ is an edge of $G'$. Indeed, let $uv\in E(G^\sharp)$ with $u,v\notin V_0$, say $u\in W_j$ and $v\in W_i$. By the definition of $G^\sharp$ in~\ref{D:deg} we have $i\neq j$, the pair $(W_i,W_j)$ is $\eps'$-regular, and $d(W_i,W_j)\ge d^\sharp$. By Claim~\ref{clm:enlargepair} and~\eqref{eq:regenlargeconst}, the pair $(V_i,V_j)$ is therefore $\eps$-regular of density at least $d^\sharp-2\delta=d$. As $u\in V_j$ and $v\in V_i$ with $i\neq j$, none of the three bullets in the statement of the theorem applies to $uv$, and so indeed $uv\in E(G')$.

Now let $v\in V(G)\setminus Z$. All the neighbours of $v$ in $G^\sharp$, with the possible exception of those lying in $V_0$, are then neighbours of $v$ in $G'$, and hence
\[
\deg_{G'}(v)\ge\deg_{G^\sharp}(v)-|V_0|\gBy{\ref{D:deg}}\deg_G(v)-(d^\sharp+\eps')n-\eps' n=\deg_G(v)-(d+2\delta+2\eps')n\;.
\]
Since $2\delta+2\eps'\le\eps^2\le\eps\le\tfrac{d}{2}$, this gives $\deg_G(v)<\deg_{G'}(v)+2dn$, as required.
\end{proof}

For $\eps,d>0$ and given a graph $G$ together with a family $\{V_i\}_{i\in[\ell]}$ of pairwise disjoint nonempty subsets of $V(G)$, which need not cover $V(G)$, we say that a weighted graph $R$ is the \emph{$(\eps,d)$-reduced graph} of $G$ with respect to $\{V_i\}_{i\in[\ell]}$ if $V(R)=[\ell]$ and $ij\in E(R)$ if and only if $(V_i,V_j)$ is an $\eps$-regular pair of density at least $d$. 
Furthermore, we define the weight of each edge $ij\in E(R)$ by $d_G(V_i,V_j)$. 
With a suitable representation, the reduced graph is close to the original graph in the cut norm distance.

\begin{fact}\label{fact:distredgraph}
Let $G$ be a graph with an $(\eps,d)$-reduced graph $R$ corresponding to an~$\eps$-regular partition $V(G)=V_1\sqcup\ldots\sqcup V_\ell$ into $\ell\ge \frac1\eps$ clusters.
Then for every graphon representation $W_G$ of $G$ there exists a graphon representation $W_R$ of $R$ such that
$$\cutn{W_G-W_R}\le 10\eps+d\;.$$
\end{fact}
\begin{proof}
We may assume that $\eps\le\frac12$, since otherwise $10\eps+d>1\ge\cutn{W_G-W_R}$ and there is nothing to prove. Write $n:=v(G)$, let $\Omega=\bigsqcup_{v\in V(G)}\Omega_v$ be the partition underlying $W_G$, and set $\Omega_{V_i}:=\bigcup_{v\in V_i}\Omega_v$. By~\ref{R:1} we have $\mu(\Omega_{V_i})=\frac{|V_i|}{n}=\frac1\ell\pm\frac\eps\ell$ for every $i\in[\ell]$. Hence we may choose a partition $\Omega=\bigsqcup_{i\in[\ell]}\Omega_i^R$ into sets of measure $\frac1\ell$ each with $\mu\bigl(\bigcup_{i\in[\ell]}(\Omega^R_i\setminus \Omega_{V_i})\bigr)\le\eps$, and we let $W_R$ be the graphon representation of $R$ given by this partition. Now, the difference $W_G-W_R$ is bounded by~1 in absolute value, and it is nonzero only on the following sets. First, on the rows and the columns indexed by $\bigcup_{i}(\Omega^R_i\setminus \Omega_{V_i})$, of total measure at most $2\eps$. Second, on the diagonal blocks $\Omega_{V_i}\times \Omega_{V_i}$, of total measure at most $\max_i\mu(\Omega_{V_i})\le\frac{1+\eps}{\ell}\le \tfrac32\eps$. Third, on the blocks corresponding to the at most $\eps\ell^2$ irregular pairs; as each such pair $\{i,j\}$ accounts for the two blocks $\Omega_{V_i}\times \Omega_{V_j}$ and $\Omega_{V_j}\times \Omega_{V_i}$, their total measure is at most $2\eps\ell^2\cdot\bigl(\frac{1+\eps}\ell\bigr)^2\le \tfrac92\eps$. Fourth, on the blocks corresponding to regular pairs of density less than $d$, where the integral of $|W_G-W_R|$ is at most $d$. Last, on the blocks corresponding to regular pairs of density at least $d$: here the $\eps$-regularity gives that for every $S,T\subseteq\Omega$ the contribution of the block $\Omega_{V_i}\times \Omega_{V_j}$ to $\int_{S\times T}(W_G-W_R)$ is at most $\eps\mu(\Omega_{V_i})\mu(\Omega_{V_j})$ in absolute value, which sums to at most $\eps$. Altogether, $\cutn{W_G-W_R}\le 9\eps+d$.
\end{proof}

The next fact is an easy consequence of the so-called counting lemma (see, e.g., \cite[Theorem 2.1]{regularity-oldsurvey}).
It follows from the fact that every walk of odd length is contained in the blow-up of a shorter walk of odd length.

\begin{fact}\label{fact:pathcounter}
    Suppose that~$\eps\ll d$, that $\sigma\in(0,1]$, and that $\frac1n\ll\zeta \ll \frac1r,\eps,\sigma$. Suppose that~$G$ is an $n$-vertex graph and that~$R$ is its~$(\eps,d)$-reduced graph with respect to a family of~$r$ pairwise disjoint sets $V_1,\ldots,V_r\subseteq V(G)$, each of size at least~$\sigma n$.
    If~$R$ contains a walk $Q$ of odd length~$\ell\le 2r-1$ then, for each odd $L\in [\ell,2r-1]$, $G$ contains at least~$\zeta n^{L+1}$ paths of length~$L$. Further, we may require that the first vertex of each such path is in the cluster represented by the first vertex of $Q$, and the last vertex is in the cluster represented by the last vertex of $Q$.
\end{fact}

The following lemma follows easily from Fact~\ref{fact:pathcounter}.

\begin{lem}\label{lem:path-oddcycle}
    Let~$\eps\ll d$, let $\sigma\in(0,1]$ and let $\frac1n\ll\zeta\ll \frac1r,\eps,\sigma$. Let~$G$ be an $n$-vertex graph and let~$R$ be its~$(\eps,d)$-reduced graph with respect to a family $\{V_i\}_{i\in[r]}$ of pairwise disjoint subsets of $V(G)$, each of size at least $\sigma n$.
    Suppose that~$R$ is connected and non-bipartite.
    Then there is an odd~$L\le 2r-1$ such that for every~$i,j\in V(R)$, not necessarily distinct, and every two sets~$X\subseteq V_i$ with~$|X|\ge \tfrac{d}2|V_i|$ and~$Y\subseteq V_j$ with~$|Y|\ge \tfrac{d}2|V_j|$, there are at least~$\zeta n^{L+1}$ paths of length~$L$ from~$X$ to $Y$.
\end{lem}

\begin{proof}
For each $i,j\in V(R)$, let $Q_{i,j}$ be a walk of odd length $\ell_{i,j}\le 2r-1$ from $i$ to $j$ in $R$ as provided by Lemma~\ref{lem:closedwalks}. Let $L:=2r-1$. 
    
    Given vertices~$i,j\in V(R)$, not necessarily distinct, and two sets~$X\subseteq V_i$ with~$|X|\ge \tfrac{d}2|V_i|$ and~$Y\subseteq V_j$ with~$|Y|\ge \tfrac{d}2|V_j|$, we first pass to disjoint sets $X'\subseteq X$ and $Y'\subseteq Y$ with~$|X'|\ge \tfrac{d}5|V_i|$ and~$|Y'|\ge \tfrac{d}5|V_j|$. For~$i\neq j$ we may simply take~$X':=X$ and~$Y':=Y$, while for~$i=j$ such a splitting is possible since~$|X|,|Y|\ge \tfrac{d}2|V_i|$ and~$|V_i|$ is large.
    Now consider the family of clusters~$\{V_k\}_{k\neq i,j}\cup \{X'\}\cup \{Y'\}$, where we `replace' the clusters~$V_i$ and~$V_j$ with the sets~$X'$ and~$Y'$ respectively, and let~$\hat R$ be the~$(5\eps/d, d/2)$-reduced graph of~$G$ with respect to this family. Its members are pairwise disjoint subsets of~$V(G)$, and each of them has at least~$\tfrac{d}{5}\sigma n$ vertices.
    Because of Fact~\ref{fact:regpairinherit} applied with~$\alpha:=\tfrac d5$, every pair which is~$\eps$-regular of density at least~$d$ in the original partition gives rise to a pair which is~$5\eps/d$-regular of density at least~$d/2$ in the new collection. Hence~$\hat R$ contains~$R$ as a subgraph, with the vertices corresponding to~$X'$ and~$Y'$ playing the roles of~$i$ and~$j$; in the case~$i=j$, the graph~$\hat R$ contains the graph obtained from~$R$ by duplicating the vertex~$i$ into two nonadjacent twins. In either case, $\hat R$ contains a walk of odd length~$\ell_{i,j}$ which starts in the vertex corresponding to~$X'$ and ends in the vertex corresponding to~$Y'$.
    In particular, Fact~\ref{fact:pathcounter} applied to $G$, to~$\hat R$ and to this walk yields the desired number of paths of length~$L$. Note that the hierarchy required there is available: the family has at most~$r+1$ members, $\tfrac{5\eps}{d}\ge\eps$, and $\tfrac{d}{5}\sigma\ge\tfrac{\eps\sigma}{5}$ because $\eps\le d$, so $\frac1n\ll\zeta\ll\frac1r,\eps,\sigma$ gives $\frac1n\ll\zeta\ll\frac1{r+1},\tfrac{5\eps}{d},\tfrac{d}{5}\sigma$; moreover $\tfrac{5\eps}{d}\ll\tfrac d2$ since $\eps\ll d$.
\end{proof}

We will also make use of the known fact that a regular pair with positive edge-density contains an almost spanning path. This fact follows in a straightforward way from the blow-up lemma~\cite{MR1684627} (or by an easy self-contained sequential procedure).

\begin{fact}\label{fact:almsppathinregp}
Let $(A,B)$ be an $\eps$-regular pair of density at least $d$ in a graph $G$, where $d>2\eps$ and~$|A|= |B|$.
Then the bipartite graph $G[A,B]$ contains a path covering all but at most $\eps |A|$ many vertices from $A$ and all but at most $\eps |B|$ many vertices from $B$.
\end{fact}

\section{Incorporating low-degree vertices into a path system}\label{sec:incorporatinglowvertices}

Low-degree vertices require special treatment in our construction of a Hamilton cycle, as they cannot be handled directly by the regularity method. Proposition~\ref{prop:lowdegreePaths} produces a \emph{path system}---that is, a collection of vertex-disjoint paths in $\G(n,W)$---covering all such vertices; in the proof of Theorem~\ref{thm:mainHC}, this path system will then be incorporated into the bulk of the Hamilton cycle constructed via the regularity method.

\begin{prop}\label{prop:lowdegreePaths}
Suppose that $W:\Omega^2\to[0,1]$ is a graphon that satisfies the condition of Theorem~\ref{thm:mainHC}\ref{en:mainHC2}. Then for every $\eps\in (0,\frac13)$ there exists $\alpha_0>0$ such that for every $\alpha\in(0,\alpha_0/2)$ and for each $n>n_0(W,\eps,\alpha)$ with probability at least $1-\eps$ the random graph $G=\G(n,W)$ contains a system $\cP$ of vertex-disjoint paths with the following properties:
\begin{enumerate}[label=(\alph*)]
\item\label{it:prop-lowdeg-atleast3} each path of $\cP$ contains at least~3 vertices,
\item \label{it:prop-lowdeg-lowdeg} every vertex of degree less than $\alpha n$ is contained in some path of $\cP$,
\item \label{it:prop-lowdeg-ends} the endvertices of each path have degree at least $\alpha n$,
\item \label{it:prop-lowdeg-VPsmall} we have $|V(\cP)|<\alpha\eps n$, and
\item \label{it:prop-lowdeg-fewpaths} we have $|\cP|<2/\alpha$.
\end{enumerate}
\end{prop}

\begin{proof}
Suppose that $\eps\in (0,\frac13)$ is given.
We use the condition of Theorem~\ref{thm:mainHC}\ref{en:mainHC2} to find $\alpha_0>0$ so that for every $\beta\in(0,\alpha_0]$ we have 
\begin{equation}\label{eq:Ma}
\mu(\SMALLDEG_W(\beta))< \eps^2\beta/999
\;.
\end{equation}
Let $\alpha\in(0,\alpha_0/2)$ be fixed. Suppose that $n$ is sufficiently large. Let $j^*\in \N$ be such that 
\begin{equation}\label{eq:such}
    2^{-j^*}\alpha_0\in\left(\frac{100}{\eps n},\frac{200}{\eps n}\right]\;.
\end{equation} 
Such a $j^*$ exists because consecutive numbers $2^{-j}\alpha_0$ differ by a factor of~2, and it satisfies $j^*\ge 1$ as soon as $n\ge \frac{200}{\eps\alpha_0}$. We have $j^*=\Theta(\log n)$. 

Throughout much of the proof, we will work with the random types $\{X_i\}_{i\in[n]}$ from~\ref{G1}.
In the first stage of the proof, we get control on the random degrees of $G$. In the easy Claim~\ref{cl:degconc} below, we will relate $\deg_G(i)$ to $\deg_W(X_i)$ up to an additive error $0.01\alpha n$. For those $i$ for which $\deg_W(X_i)$ is small, a more precise control is needed. This is done in Claims~\ref{cl:smallvert}, \ref{cl:Ebig}, and~\ref{cl:R}. In the second stage of the proof, we use the gained control on the degrees and deterministically build the path system $\cP$.
\begin{claim}\label{cl:degconc}
    Let $\mathcal{K}$ be the event that for each $i\in[n]$ we have $\deg_G(i)=\deg_W(X_i)n\pm 0.01\alpha n$. Then we have $\Probability[\mathcal{K}]\ge 1-\frac{\eps}{3}$.
\end{claim}
\begin{claimproof}
    This follows immediately from Fact~\ref{fact:sampledeg}, applied with $\min\{0.01\alpha,\eps/3\}$ in place of $\eps$.
\end{claimproof}

\begin{claim}\label{cl:smallvert}
Let $\mathcal{T}$ be the event that $\{X_i:i\in[n]\}\cap \SMALLDEG_W(2^{-j^*}\alpha_0)=\emptyset$. Then we have $\Probability[\mathcal{T}]\ge 1-\frac{\eps}{4.8}$.   
\end{claim}
\begin{claimproof}
Indeed, \[
\Probability[\mathcal{T}]=(1-\mu(\SMALLDEG_W(2^{-j^*}\alpha_0)))^n\geBy{\eqref{eq:Ma},\eqref{eq:such}}(1-\tfrac{200\eps}{999n})^n\ge \exp(-\eps/4.8)\ge 1-\eps/4.8
\;,
\]
as was needed. The penultimate inequality uses that $n$ is large, together with $\tfrac{200}{999}<\tfrac{5}{24}=\tfrac1{4.8}$.
\end{claimproof}
Let $\cE$ be the event that $\mathcal{T}$ holds and that for every $\beta\in (0,\alpha_0]$ we have $|\{i\in [n]: X_i\in  \SMALLDEG_W(\beta)\}|< \eps \beta n/20$. 

\begin{claim}\label{cl:Ebig}
 We have $\Probability[\cE]\ge 1-\frac{\eps}{3}$.   
\end{claim}
\begin{claimproof}
 For each $k\in \{0,\ldots,j^*-1\}$, let $\cA_k$ be the event that $|\{i\in [n]: X_i\in \SMALLDEG_W(2^{-k}\alpha_0)\}|\ge 2^{-(k+1)}\eps \alpha_0 n/20$. We claim that $\mathcal{T}\setminus\bigcup_{k=0}^{j^*-1}\cA_k\subseteq \cE$. Indeed, suppose that $\mathcal{T}$ holds and that none of the events $\cA_k$ does, and let $\beta\in(0,\alpha_0]$ be arbitrary. If $\beta\le 2^{-j^*}\alpha_0$ then $\SMALLDEG_W(\beta)\subseteq \SMALLDEG_W(2^{-j^*}\alpha_0)$ contains no point $X_i$ by $\mathcal{T}$. Otherwise, there is $k\in \{0,\ldots,j^*-1\}$ with $2^{-(k+1)}\alpha_0\le \beta\le 2^{-k}\alpha_0$, and as $\cA_k$ does not hold, we get $|\{i\in [n]: X_i\in \SMALLDEG_W(\beta)\}|<2^{-(k+1)}\eps \alpha_0 n/20\le \eps\beta n/20$, as required. Consequently, $\Probability[\cE]\ge \Probability[\mathcal{T}]-\sum_{k=0}^{j^*-1}\Probability[\cA_k]\ge 1-\frac{\eps}{4.8}-\sum_{k=0}^{j^*-1}\Probability[\cA_k]$ by Claim~\ref{cl:smallvert}. 

 Fix $k\in \{0,\ldots,j^*-1\}$. The number of indices $i\in[n]$ for which $X_i\in \SMALLDEG_W(2^{-k}\alpha_0)$ is binomial with parameters $n$ and $\mu(\SMALLDEG_W(2^{-k}\alpha_0))\lByRef{eq:Ma} 2^{-k}\eps^2\alpha_0/999$. As the binomial distribution is stochastically monotone in its parameter, this random variable is dominated by a binomial random variable with parameters $n$ and $2^{-k}\eps^2\alpha_0/999$. The Chernoff bound applied to the latter (Lemma~\ref{lem:Chernoff}\ref{it:Chernoff2} with $C=999/(40\eps)$) gives $\Probability[\cA_k]\le \exp(-\tfrac{2^{-k}\eps^2\alpha_0}{999}\cdot \frac{999}{40\eps}\ln(\tfrac{999}{40e\eps})\cdot n)\le\exp(-\tfrac{2^{-k}\eps\alpha_0}{40}\cdot \ln(\tfrac{9}{\eps})\cdot n)$. For $k=j^*-1$ note that $\tfrac{200}{\eps n}\le 2^{-(j^*-1)}\alpha_0$ by~\eqref{eq:such}, hence the previous bound gives $\Probability[\cA_{j^*-1}]\le\exp(-\tfrac{5}{n}\cdot \ln(\tfrac{9}{\eps})\cdot n)\le(\frac{\eps}{9})^5$. Therefore for every $k\in \{0,\ldots,j^*-1\}$ we have
 \[\exp\big(-\tfrac{2^{-k}\eps\alpha_0}{40}\cdot \ln(\tfrac{9}{\eps})\cdot n\big)
 \le 
 \exp\big(-\tfrac{2^{-(j^*-1)}\eps\alpha_0}{40}\cdot \ln(\tfrac{9}{\eps})\cdot n\big)^{2^{j^*-k-1}}\le2^{k-j^*}\cdot\tfrac{\eps}{9}\;.\]
 Thus $\sum_{k=0}^{j^*-1}\Probability[\cA_k]\le\sum_{i=1}^\infty 2^{-i}\cdot\frac{\eps}{9}\le \frac{\eps}{8}$ proving the claim.
 \end{claimproof}

For $i\in[n]$, let $\cR_i$ be the event that $X_i\in \SMALLDEG_W(\alpha_0)\setminus \SMALLDEG_W(2^{-j^*}\alpha_0)$ and at the same time $\deg_G(i)<\deg_W(X_i)(n-1)/2$.
\begin{claim}\label{cl:R}
We have $\Probability[\bigcup_{i\in [n]}\cR_i]\le \eps/3$.
\end{claim}
\begin{claimproof}
Observe that for $\cR_i$ to occur, we have to have at least one index $j\in [j^*]$ for which $X_i\in \SMALLDEG_W(2^{-(j-1)}\alpha_0)\setminus \SMALLDEG_W(2^{-j}\alpha_0)$ and $\deg_G(i)<\deg_W(X_i)(n-1)/2$. Observe that conditioning on $X_i=x$ (subject to $x\in \SMALLDEG_W(2^{-(j-1)}\alpha_0)\setminus \SMALLDEG_W(2^{-j}\alpha_0)$), we have that $\deg_G(i)$ is a binomial random variable with parameters $n-1$ and $\deg_W(x)$. Hence, the Chernoff bound (Lemma~\ref{lem:Chernoff}\ref{it:Chernoff1}, applied with $n-1$ in place of $n$) tells us that 
\[
\Probability[\deg_G(i)<\deg_W(x)(n-1)/2\:|\:X_i=x]\le \exp(-(n-1)\deg_W(x)/8)\le \exp(-2^{-(j-1)}\alpha_0 n/20)
\;,
\]
where we used that $\deg_W(x)> 2^{-j}\alpha_0$ and that $n\ge 5$.
Define a function $F:[j^*]\to (0,1]$ by $F(j)=\exp(-2^{-(j-1)}\alpha_0 n/20)$. Thus, for any $i\in [n]$, we have
\[
\Probability[\cR_i]\le 
\sum_{j=1}^{j^*}\mu(\SMALLDEG_W(2^{-(j-1)}\alpha_0))F(j)
\leByRef{eq:Ma}
\sum_{j=1}^{j^*}\underbrace{\frac{\eps^2\cdot2^{-(j-1)}\alpha_0}{999}}_{\mathsf{(T1)}_j}\cdot \underbrace{F(j)}_{\mathsf{(T2)}_j}\;.
\]
Let us look at the terms $\mathsf{(T1)}_j$ and $\mathsf{(T2)}_j$, in the descending order, that is, as $j$ moves from~$j^*$ to~$1$. For each $j$ in this range, we have $\frac{\mathsf{(T1)}_{j-1}}{\mathsf{(T1)}_{j}}=2$. Further, as $F$ is increasing, we have $\frac{\mathsf{(T2)}_{j-1}}{\mathsf{(T2)}_{j}}=F(j)\le F(j^*)=\exp(-2^{-j^*}\alpha_0\cdot 2n/20)\leByRef{eq:such}\exp(-10/\eps)<\frac{1}{18}$. Hence, when $j$ decreases by~1, the corresponding term of the sum gets multiplied by $2F(j)\le 2\cdot\frac1{18}=\frac19$; that is, it shrinks at least ninefold. Thus, the sum $\sum_{j=1}^{j^*}\frac{\eps^2\cdot2^{-(j-1)}\alpha_0}{999}\cdot F(j)$ is bounded above by a geometric series with initial term 
\[\frac{\eps^2\cdot2^{-(j^*-1)}\alpha_0}{999}\cdot F(j^*)
\le \frac{2\eps^2\cdot2^{-j^*}\alpha_0}{999}\cdot\frac12\leByRef{eq:such} \frac{200\eps}{999n}\le\frac{2\eps}{9n}
\]
and common ratio $\frac{1}{9}$. We conclude that $\Probability[\cR_i]\le \eps/(4n)$. The union bound gives $\Probability[\bigcup_{i\in [n]}\cR_i]\le \eps/3$.
\end{claimproof}

We shall show that if $\mathcal{K}\cap \cE\setminus \bigcup_{i\in [n]}\cR_i$ holds, then we get the assertions of the proposition. By Claims~\ref{cl:degconc}, \ref{cl:Ebig}, and~\ref{cl:R}, the probability that this fails is at most $\frac\eps3+\frac\eps3+\frac\eps3=\eps$, as required. 
The rest of the proof is hence devoted to the construction of the path system $\cP$.

Let $t:=|\{i\in [n]: X_i\in \SMALLDEG_W(2\alpha)\}|$. Let us enumerate the elements $\{i\in [n]:X_i\in \SMALLDEG_W(2\alpha)\}$ as $i_1,i_2,\ldots,i_t$ so that $\deg_W(X_{i_1}),\deg_W(X_{i_2}),\ldots,\deg_W(X_{i_t})$ is nondecreasing.

Inductively for $r=1,\ldots,t$, we will fix an arbitrary set $U_r$ of two vertices with the property that $U_r\subseteq N_G(i_r)\setminus \left(\bigcup_{s=1}^{r-1} U_s\cup\{i_1,\ldots,i_{r-1}\}\right)$. Let us argue that this is possible at a given step $r$. It suffices to show that 
\begin{equation}
    \label{eq:ss}
\deg_G(i_r)\ge 2+3(r-1)\;.
\end{equation}
We can use the defining property of $\cE$ with $\beta:=\deg_W(X_{i_r})$. We get that $r<\frac{\eps\deg_W(X_{i_r})n}{20}$. Since the event $\cR_{i_r}$ does not hold, while $\mathcal{T}$ does, we have $\deg_G(i_r)\ge \deg_W(X_{i_r})(n-1)/2$. Putting these two bounds together and using $\eps<\frac13$ and $n\ge5$, we get $\deg_G(i_r)\ge\frac{\deg_W(X_{i_r})(n-1)}{2}>\frac{20r}{\eps n}\cdot\frac{n-1}{2}>24r\ge 2+3(r-1)$, which is~\eqref{eq:ss}. This allows us to choose $U_r$.

We have just shown that whenever we have $\cE\setminus \bigcup_{i\in [n]}\cR_i$, we can create a system of vertex-disjoint binary trees that covers all vertices $i_1,\ldots,i_t$ and some additional vertices. Indeed, for each vertex $i_r$ (which is either an internal vertex of a binary tree or the root, depending on whether $i_r\in \bigcup_{s=1}^{r-1}U_s$ or not), the vertices of $U_r$ specify the two descendants of $i_r$. Every vertex outside $\{i_1,\ldots,i_t\}$ is a leaf. Let us verify that this is a system of binary trees in the sense of the definition preceding Fact~\ref{fact:decomposetree}. The sets $U_1,\ldots,U_t$ are pairwise disjoint, so every vertex $u\notin\{i_1,\ldots,i_t\}$ lies in exactly one of them and has degree~1. Each $i_r$ is adjacent to the two vertices of $U_r$, and to at most one further vertex, namely to $i_s$ if $i_r\in U_s$ for some $s$; such an $s$ is unique and satisfies $s<r$, because $U_s$ is disjoint from $\{i_1,\ldots,i_{s-1}\}$ and $i_s\notin N_G(i_s)$. In particular, passing from a vertex to its parent strictly decreases the index, so there is no cycle. Thus each $i_r$ has degree~3 if it has a parent and degree~2 otherwise, and in every component exactly one vertex is of the latter kind, namely the $i_r$ of smallest index in that component. We can now turn the system of binary trees into a system of paths $\mathcal{F}$ by applying Fact~\ref{fact:decomposetree} to each of the trees separately; recall that $\mathcal{F}$ covers all the non-leaf vertices of these trees, while some of their leaves may be left uncovered. Take an arbitrary path in $\mathcal{F}$. All its vertices $i$ except its ends satisfy $X_i\in \SMALLDEG_W(2\alpha)$. At the same time, as the path has at least~3 vertices, at least a third of its vertices are internal. Hence,
\begin{equation}\label{eq:Vance}
    |\mathcal{F}|\le t \le  \eps \alpha n/10 \quad\mbox{and}\quad \frac{1}{3}|V(\mathcal{F})|\le t\le \eps \alpha n/10\;,
\end{equation}
where we used the definition of~$\cE$ to get an upper bound on $t$.

The system $\mathcal{F}$ could almost serve as $\cP$ from the statement of the proposition; in particular it satisfies~\ref{it:prop-lowdeg-atleast3}, \ref{it:prop-lowdeg-lowdeg}, and~\ref{it:prop-lowdeg-ends}. The reason is that for every vertex $i$ of degree less than $\alpha n$ in $G$ we have $\deg_W(X_i)n<\alpha n+0.01\alpha n$ by Claim~\ref{cl:degconc}, so its type lies in $\SMALLDEG_W(2\alpha)$, and $i$ is hence included as an internal vertex in $\mathcal{F}$. Similarly, since the end vertices have their types not in $\SMALLDEG_W(2\alpha)$, by Claim~\ref{cl:degconc} they have degree at least $2\alpha n-0.01\alpha n=1.99\alpha n\ge\alpha n$ in $G$.

It remains to show how to modify it in order to get~\ref{it:prop-lowdeg-VPsmall} and~\ref{it:prop-lowdeg-fewpaths}.

Among all path systems $\cP$ satisfying properties~\ref{it:prop-lowdeg-atleast3}, \ref{it:prop-lowdeg-lowdeg}, and~\ref{it:prop-lowdeg-ends} of the proposition, and additionally obeying
\begin{equation}\label{eq:pracka}
    |V(\cP)| \le |V(\mathcal{F})| + |\mathcal{F}| - |\cP| \,,
\end{equation}
we choose $\cP$ that minimizes $|\cP|$. 
The role of~\eqref{eq:pracka} is twofold: it bounds $|V(\cP)|$, which is what~\ref{it:prop-lowdeg-VPsmall} asks for, and it is preserved by the only operation we shall perform on $\cP$ below, namely merging two of its paths through a vertex outside $V(\cP)$.
Observe that this minimization is over a nonempty solution space, since~$\mathcal{F}$ itself satisfies all the required conditions. We claim the remaining properties~\ref{it:prop-lowdeg-VPsmall} and~\ref{it:prop-lowdeg-fewpaths} are then satisfied as well.

Let us first look at~\ref{it:prop-lowdeg-VPsmall}. By~\eqref{eq:pracka}, we have $|V(\cP)| \le |V(\mathcal{F})| + |\mathcal{F}|$. The desired bound follows by plugging in~\eqref{eq:Vance}.
    
We now look at~\ref{it:prop-lowdeg-fewpaths}. Let~$U$ be a set of vertices defined by taking one endpoint for each path in~$\cP$; as each path of~$\cP$ has at least~3 vertices, its two endpoints are distinct, and so $|U|=|\cP|$.

    Observe that for every pair of distinct vertices~$u,v\in U$ we have that~$N_G(u)\setminus V(\cP)$ and $N_G(v)\setminus V(\cP)$ are disjoint. 
    Indeed, otherwise, we could merge two paths from~$\cP$ (which preserves \ref{it:prop-lowdeg-atleast3}, \ref{it:prop-lowdeg-lowdeg}, and~\ref{it:prop-lowdeg-ends}, as well as~\eqref{eq:pracka}), contradicting the minimality of $|\cP|$.
    Then,
 \begin{align*}
    n &\geBy{\text{disjoint}}  \sum_{v\in U} |N_G(v)\setminus V(\cP)| \\
  \JUSTIFY{\text{Proposition~\ref{prop:lowdegreePaths}\ref{it:prop-lowdeg-ends} and Proposition~\ref{prop:lowdegreePaths}\ref{it:prop-lowdeg-VPsmall} together with $\eps<\tfrac13$}}
    &\geBy{\phantom{\text{disjoint}}}  |U| \Big(\alpha n-\alpha n/3\Big)=|\cP| \cdot\tfrac23\alpha n\,.
 \end{align*}
    This yields~$|\cP| \le \tfrac{3}{2\alpha}<\tfrac{2}{\alpha}$ as desired.    
\end{proof}

\section{Typical structure of the cluster graph of $\G(n,W)$}\label{StructureOfTheClusterGraph}
The main part of the Hamilton cycle in our proof of Theorem~\ref{thm:mainHC} is constructed using Szemerédi's regularity lemma. In this section, we present the main structural results to this end. Namely, Proposition~\ref{prop:clustergraph} asserts that a cluster graph such as one coming from a regularization of $\G(n,W)$ has only few vertices of low weighted degree, and that after deleting these it becomes connected and uniquely half-covered---robustly so, in that both properties survive the deletion of any further small set of vertices.

\subsection{Fractional vertex covers in graphons}\label{ssec:fractionalvertexcoversgraphons}

As previously noted, matching theory---when applied to the cluster graph---plays a crucial role in our proof of Theorem~\ref{thm:mainHC}. The applicability of this approach relies on specific conditions imposed on the graphon $W$, most notably the condition of Theorem~\ref{thm:mainHC}\ref{en:mainHC3}. Consequently, we require an extension of matching theory to the graphon setting. Such a framework was developed in~\cite{MR4186624,DolezalHladky:MatchingPolytons}. 

It turns out that between the two dual aspects of matching theory---the ``matching'' facet and the ``vertex cover'' facet---the latter behaves more cleanly for graphons. We shall restrict our review to this facet. Notably, we will not require a full ``LP-type duality'' between vertex covers and matchings at the level of the graphon $W$; instead, this transition is handled at the level of the finite cluster graph $R$ of $\G(n,W)$.

Given a graphon $W:\Omega^2\to [0,1]$, we say that a function $f:\Omega\to[0,1]$ is a \emph{fractional vertex cover} of $W$ if, for almost every $(x,y)\in \Omega^2$, we have $f(x)+f(y)\ge 1$ whenever $W(x,y)>0$. As for graphs, we call $f$ \emph{half-integral} if its range is a subset of $\{0,\frac12,1\}$. 
The observation made in~\ref{corr1} regarding the correspondence between graph peninsulae and half-integral vertex covers extends naturally to graphons. We record and prove this for later reference.
\begin{fact}\label{fct:halfintvcispeninsula}
A graphon $W$ has a peninsula if and only if there is a half-integral vertex cover $f$ of $W$ with $\| f\|_1\le\frac{1}{2}$ which is not equal to $\frac{1}{2}$ almost everywhere.
\end{fact}
\begin{proof}
Suppose first that $W$ has a peninsula, witnessed by $a\in(0,\frac12]$ and by sets $A,B\subseteq\Omega$ as in Definition~\ref{defi:peninsula}\ref{en:defiPen1}. Let $f$ be 0 on $A$, $\frac12$ on $B$, and 1 on $\Omega\setminus(A\cup B)$. The only pairs $(x,y)$ with $f(x)+f(y)<1$ lie in $\left(A\times (A\cup B)\right)\cup\left((A\cup B)\times A\right)$, where $W$ is constant-0. Hence $f$ is a half-integral vertex cover of $W$. Further, $\|f\|_1=\frac12(1-2a)+a=\frac12$, and $f$ is not equal to $\frac12$ on the set $A$, which is of positive measure.

Conversely, suppose that $f$ is as in the statement and set $A:=f^{-1}(0)$ and $B:=f^{-1}(\frac12)$. As $f$ is a fractional vertex cover, $W$ is constant-0 on $A\times(A\cup B)$. Since $\|f\|_1=1-\mu(A)-\frac{\mu(B)}2$, the assumption $\|f\|_1\le\frac12$ says that
\begin{equation}\label{eq:peninsulanorm}
2\mu(A)+\mu(B)\ge 1\;.
\end{equation}
As $f$ is not equal to $\frac12$ almost everywhere, we have $\mu(B)<1$, and so~\eqref{eq:peninsulanorm} gives $\mu(A)>0$. Set $a:=\min\{\mu(A),\frac12\}\in(0,\frac12]$. We have $\mu(A)+\mu(B)\ge 1-a$; indeed, for $a=\mu(A)$ this is a restatement of~\eqref{eq:peninsulanorm}, while for $a=\frac12$ it follows from $\mu(A)>\frac12$. Since $\Omega$ is atomless, we may therefore choose sets $A^\circ\subseteq A$ and $B^\circ\subseteq (A\cup B)\setminus A^\circ$ with $\mu(A^\circ)=a$ and $\mu(B^\circ)=1-2a$. As $A^\circ\times(A^\circ\cup B^\circ)\subseteq A\times(A\cup B)$, the sets $A^\circ$ and $B^\circ$ witness that $W$ has a peninsula.
\end{proof}

The main result of this section shows that if a sequence of graphons converges in the cut norm, any sequence of associated half-integral vertex covers yields a half-integral vertex cover of the limit graphon with favourable properties. 

\begin{prop}\label{prop:weakstarlimit}
Suppose that $W_1, W_2, \ldots$ is a sequence of graphons on $\Omega$ converging to a graphon $W$ in the cut norm. Let $f_1,f_2,\ldots: \Omega \to \{0, \frac12, 1\}$ be a sequence such that $f_n$ is a half-integral vertex cover of $W_n$. Then there exists a half-integral vertex cover $f: \Omega \to \{0, \frac12, 1\}$ of $W$ such that:
\begin{enumerate}[label=(\alph*)]
    \item\label{it:wslim-zeroset} $\mu(f^{-1}(0)) \ge \liminf_{n \to \infty} \mu(f_n^{-1}(0))$, and
    \item\label{it:wslim-norm} $\|f\|_1 \le \liminf_{n \to \infty} \|f_n\|_1$.
\end{enumerate}
\end{prop}

\begin{proof}
The proof follows the spirit of Lemma~3 in \cite{DolezalHladky:MatchingPolytons} and Lemma~2.3 in \cite{MR3879960}, and uses weak* convergence as we introduced in Section~\ref{ssec:FunctionalAnalysis}. We begin with a key claim regarding the interaction of weak* limits and the graphon edge weights.

\begin{claim}\label{cl:doublyindependent}
Let $A_n, B_n \subseteq \Omega$ be sets such that $\int_{A_n \times B_n} W_n = 0$ for each $n$. Suppose $\mathds{1}_{A_n} \xrightarrow{w^*} a$ and $\mathds{1}_{B_n} \xrightarrow{w^*} b$. Then $\int_{\supp(a) \times \supp(b)} W = 0$.
\end{claim}

\begin{claimproof}
By Fact~\ref{fact:weakstarbounded}, we have $a(x),b(x)\in[0,1]$ for almost every $x\in\Omega$. Hence, up to null sets, $\supp(a)$ and $\supp(b)$ are the unions of the sets $\supp_{1/k}(a)$ and of the sets $\supp_{1/k}(b)$, $k\ge 2$, respectively. It is therefore enough to show that
\[\int_{\supp_\delta(a) \times \supp_\delta(b)} W = 0\]
for every $\delta\in(0,1)$.

Suppose for contradiction that this fails for some $\delta\in(0,1)$. Then there is $c > 0$ such that the set $E:=\{(x,y)\in \supp_\delta(a)\times\supp_\delta(b): W(x,y)>c\}$ has positive measure. Recall from Section~\ref{ssec:graphonpreliminaries} that we may assume that $\Omega=[0,1]$ is equipped with the Lebesgue measure. Write $\eta:=\frac{\delta^2}{10}$. By the Lebesgue density theorem (see, e.g.,~\cite[Chapter~3]{Folland1999}), almost every point of $E$ is a point of density~1 of $E$; taking a sufficiently small square around such a point, we obtain intervals $I,J\subseteq [0,1]$ with $\mu^2\left(E\cap (I\times J)\right)\ge (1-\eta)\mu(I)\mu(J)$. Set $X:=I\cap \supp_\delta(a)$ and $Y:=J\cap\supp_\delta(b)$. As $E\cap (I\times J)\subseteq X\times Y\subseteq I\times J$, we get $\mu(X)\mu(Y)\ge (1-\eta)\mu(I)\mu(J)>0$ and
\begin{equation}\label{eq:MiJa}
\mu^2\left(\{(x,y) \in X \times Y : W(x,y) \le c\}\right)\le \mu(I)\mu(J)-\mu^2\left(E\cap (I\times J)\right)\le\frac{\eta}{1-\eta}\mu(X)\mu(Y)\le \frac{\delta^2 \mu(X)\mu(Y)}{5} .
\end{equation}
By the definition of $\supp_\delta$, we have $\int_X a \diff\mu \ge \delta \mu(X)$ and $\int_Y b \diff\mu \ge \delta \mu(Y)$. Due to weak* convergence, for sufficiently large $n$, we have $\mu(X \cap A_n) \ge \frac{\delta}{2}\mu(X)$ and $\mu(Y \cap B_n) \ge \frac{\delta}{2}\mu(Y)$. Then \eqref{eq:MiJa} implies
\begin{equation}\label{eq:LB}
\int_{(X \cap A_n) \times (Y \cap B_n)} W \ge \left( \frac{\delta^2}{4} - \frac{\delta^2}{5} \right) c \mu(X)\mu(Y) = \frac{\delta^2 c \mu(X)\mu(Y)}{20} .
\end{equation}
However, since $W_n \to W$ in cut norm and $\int_{A_n \times B_n} W_n = 0$,
\begin{equation}\label{eq:UB}
\int_{(X \cap A_n) \times (Y \cap B_n)} W \le \int_{A_n \times B_n} W = \int_{A_n \times B_n} W_n + o(1) = o(1) .
\end{equation}
The contradiction between \eqref{eq:LB} and \eqref{eq:UB} proves the claim.
\end{claimproof}

The sequences $\left(\|f_n\|_1\right)_n$, $\left(\mu(f_n^{-1}(0))\right)_n$, and $\left(\mu(f_n^{-1}(1/2))\right)_n$ are bounded, and so we can pass to a subsequence along which all three converge and along which the limit of the first one equals its limit inferior over the original sequence. The limit of the second one along the subsequence is then at least its limit inferior over the original sequence. Hence, it suffices to prove~\ref{it:wslim-zeroset} and~\ref{it:wslim-norm} with the limits along the subsequence in place of the two limits inferior. By the Banach--Alaoglu Theorem, we can pass to a further subsequence such that the sequences $\left(\mathbf{1}_{f_n^{-1}(0)}\right)_n$, $\left(\mathbf{1}_{f_n^{-1}(1/2)}\right)_n$, and $\left(\mathbf{1}_{f_n^{-1}(1)}\right)_n$ converge weak*; from now on, all the limits are taken along this final subsequence. Since $f_n$ is a half-integral vertex cover, we have $W_n(x,y) = 0$ for almost all $(x,y) \in f_n^{-1}(0) \times (f_n^{-1}(0) \cup f_n^{-1}(1/2))$. Applying Claim~\ref{cl:doublyindependent} to $A_n := f_n^{-1}(0)$ and $B_n := f_n^{-1}(0) \cup f_n^{-1}(1/2)$ and their weak* limits $a$ and $b$, we define $f: \Omega \to \{0, \frac12, 1\}$ as
\[ f(x) = \begin{cases} 0 & x \in \supp(a) \\ \frac{1}{2} & x \in \supp(b) \setminus \supp(a) \\ 1 & x \in \Omega \setminus \supp(b)\;. \end{cases} \]
Since $A_n\subseteq B_n$, for every $X\subseteq\Omega$ we have $\int_X a\diff\mu=\lim_n \mu(X\cap A_n)\le \lim_n \mu(X\cap B_n)=\int_X b\diff\mu$, and thus $a\le b$ almost everywhere. In particular, up to a null set we have $\supp(a)\subseteq\supp(b)$, and therefore $f^{-1}(0)\cup f^{-1}(1/2)=\supp(b)$.
By Claim~\ref{cl:doublyindependent}, $f$ is a half-integral vertex cover of $W$. Finally, we verify the asserted properties. We have
\begin{equation}
\label{eq:rme}
\mu(f^{-1}(0)) = \mu(\supp(a)) = \int_{\supp(a)} 1 \diff\mu(x) \geBy{Fact~\ref{fact:weakstarbounded}}\int_{\supp(a)} a(x) \diff\mu(x)=\int_{\Omega} a(x) \diff\mu(x)= \lim_{n \to \infty} \mu(f_n^{-1}(0)) ,
\end{equation}
which establishes~\ref{it:wslim-zeroset}, and is also one of two useful inequalities towards~\ref{it:wslim-norm}. The other inequality is obtained by repeating~\eqref{eq:rme} for $\mu(f^{-1}(0) \cup f^{-1}(1/2))=\mu(\supp(b))$ in place of $\mu(f^{-1}(0))=\mu(\supp(a))$, with $b$ in place of $a$. We obtain $\mu(f^{-1}(0) \cup f^{-1}(1/2)) \ge \lim_n (\mu(f_n^{-1}(0)) + \mu(f_n^{-1}(1/2)))$. These two inequalities give
\begin{align*}
\|f\|_1 &= 1 - \frac{\mu(f^{-1}(0) \cup f^{-1}(1/2))}{2} - \frac{\mu(f^{-1}(0))}{2} \\
&\le 1 - \frac{1}{2} \lim_{n \to \infty} \left( \mu(f_n^{-1}(0)) + \mu(f_n^{-1}(1/2)) \right) - \frac{1}{2} \lim_{n \to \infty} \mu(f_n^{-1}(0)) = \lim_{n \to \infty} \|f_n\|_1 ,
\end{align*}
as was needed for~\ref{it:wslim-norm}.
\end{proof}

\subsection{Fractional matchings in the cluster graph}
Recall that given $\delta>0$ and a graphon $W$ we write $\SMALLDEG_{W}(\delta):=\{x\in \Omega\colon \deg_W(x)\le\delta\}$, and that in Section~\ref{ssec:graphonpreliminaries} we introduced the analogous set $\SMALLDEG_{G}(\delta)$ for a weighted graph $(G,w)$, defined through the weighted degree, as well as the notion of the cut distance between a weighted graph and a graphon.

The next proposition serves as the bridge that transfers conditions~\ref{en:mainHC1}, \ref{en:mainHC2}, and \ref{en:mainHC3} of Theorem~\ref{thm:mainHC} to structural properties of a regularization $R$ of $\G(n,W)$. More precisely, it states that after deleting from $R$ its few vertices of low weighted degree, the resulting graph is connected and uniquely half-covered. Moreover, these properties persist even after the adversarial deletion of an additional small set of vertices. 

\begin{prop}\label{prop:clustergraph}
Suppose that $W:\Omega^2\to[0,1]$ is a graphon that satisfies conditions~\ref{en:mainHC1}, \ref{en:mainHC2}, and \ref{en:mainHC3} of Theorem~\ref{thm:mainHC}.
Then for every $C>1$ and $\gamma\in(0,\tfrac{1}{20C})$ there exists $\alpha_0\in(0,1]$ such that the following holds for every $\alpha\in(0,\alpha_0]$. Let $R$ be a weighted graph of order $n$ whose cut distance from $W$ is less than $\tfrac{\gamma\alpha^2}{8}$. Then, with $R^*:=R-\SMALLDEG_R(\alpha)$, the following hold, the last two of them for every set $A\subseteq V(R^*)$ with $|A|\le C\gamma \alpha n$:
\begin{enumerate}[label=(\alph*)]
    \item\label{en:size} we have $|V(R)\setminus V(R^*)|=|\SMALLDEG_{R}(\alpha)|\le \gamma\alpha n$,
    \item\label{en:halfcov} $R^*-A$ is uniquely half-covered, and 
    \item\label{en:connected}~$R^*-A$ is connected.
\end{enumerate}
\end{prop}
In our proof of Theorem~\ref{thm:mainHC}, we apply Proposition~\ref{prop:clustergraph} to two regularizations of $\G(n,W)$, each serving a different but complementary role.
The first regularization (denoted in Section~\ref{sec:proof} by $R$, with error parameter $\eps_1$) is used to produce many odd cycles, which, via Lemma~\ref{lem:abs}, yield an absorbing path. Here, the only used consequence of Proposition~\ref{prop:clustergraph}\ref{en:halfcov} is that $R^*-A$ is non-bipartite (Lemma~\ref{lem:OurCombOpt}\ref{en:OurCombOptNonBip}).
The second regularization (denoted by $R_2$, with error parameter $\eps_2\ll\eps_1$) is used to construct an almost perfect fractional matching in $R_2$, which is then exploited via a blow-up-lemma-type argument to embed paths forming the bulk of the Hamilton cycle. In this step, non-bipartiteness is irrelevant, while the existence of a (near-)perfect fractional matching is essential (Lemma~\ref{lem:OurCombOpt}\ref{en:OurCombOptPerfect}).

\begin{proof}[Proof of Proposition~\ref{prop:clustergraph}]
Throughout, we refer to conditions \ref{en:mainHC1}, \ref{en:mainHC2}, and \ref{en:mainHC3} of Theorem~\ref{thm:mainHC} simply as \ref{en:mainHC1}, \ref{en:mainHC2}, and \ref{en:mainHC3}. Also, we shall work with a sequence of weighted graphs $R_i$ ($i=1,2,\ldots$) and their suitably defined induced subgraphs $R_i^*$ and $R_i'$. We view $R_i^*$ and $R_i'$ also as weighted, with the weights inherited from $R_i$; their weighted degrees, and the sets $\SMALLDEG_{\cdot}(\delta)$ derived from them, always refer to these inherited weights.

Suppose that the assertion of the proposition does not hold for $W$ and constants $C>1$ and $\gamma\in(0,\tfrac{1}{20C})$. 
By~\ref{en:mainHC2} we may choose $\beta_0>0$ such that
\begin{equation}\label{eq:smalldegW}
\mu(\SMALLDEG_{W}(2\delta))<\gamma\delta/2\textrm{ for every }\delta\in(0,\beta_0]\;.
\end{equation}
Let $\kappa$ be given by Proposition~\ref{prop:connectedgraphon}, applied by~\ref{en:mainHC1} to $W$ with $\zeta:=\beta_0/10$.
Applying the negated assertion with $\tfrac1i$ in place of $\alpha_0$ for each $i\in \N$, we obtain a sequence $\delta_1,\delta_2,\ldots>0$ with $\delta_i\le \tfrac1i$, and thus $\lim_{i\rightarrow\infty}\delta_i=0$, with the following property. Setting
\begin{equation}\label{eq:cutdistdeltai}
\eps_i:=\tfrac{\gamma\delta_i^2}{8}\;,
\end{equation}
for each $i\in \N$ there exists a weighted graph $R_i$, say of order $n_i$, with a graphon representation $W_{R_i}$ such that
\begin{equation}\label{eq:cutnormdist}
\cutn{W-W_{R_i}}<\eps_i
\end{equation}
and such that, with $\delta_i$ in the role of $\alpha$ and with $R_i^*:=R_i-\SMALLDEG_{R_i}(\delta_i)$, at least one of the properties~\ref{en:size}--\ref{en:connected} fails. That is, either~\ref{en:size} fails, or there exists a set $A\subseteq V(R_i^*)$ with $|A|\le C\gamma\delta_i n_i$ for which~\ref{en:halfcov} or~\ref{en:connected} fails.  
Choose $i_0\in\N$ such that
\begin{equation}\label{eq:i0}
4C\delta_i<\kappa\textrm{ and }\delta_i<\beta_0\textrm{ for every }i\ge i_0\;.
\end{equation}
We make the following observation.
\begin{claim}\label{clm:vanashRi}
    For every $i\ge i_0$ and $\delta\in[\delta_i/2,\beta_0]$ we have $|\SMALLDEG_{R_i}(\delta)|<\gamma\delta n_i$. 
\end{claim}

\begin{claimproof}
Assume for the sake of contradiction that $|\SMALLDEG_{R_i}(\delta)|\ge\gamma\delta n_i$ for some $\delta_i/2\le\delta\le\beta_0$ and $i\ge i_0$. By~\eqref{eq:wdegrepr}, this gives $\mu(\SMALLDEG_{W_{R_i}}(\delta))\ge \gamma\delta$. Set $S:=\SMALLDEG_{W_{R_i}}(\delta)\setminus \SMALLDEG_W(2\delta)$. Combined with~\eqref{eq:smalldegW}, we have $\mu(S)\ge \frac{\gamma\delta}2$. Also, observe that for every $x\in S$, we have $\int_{y\in \Omega} \left(W(x,y)-W_{R_i}(x,y)\right)\diff\mu(y)>\delta$. This gives
\[\cutn{W-W_{R_i}}\ge\left|\int_{S\times\Omega}\left(W-W_{R_i}\right)\diff \mu^2\right|>\frac{\gamma\delta}2\cdot \delta=\frac{\gamma\delta^2}2\ge \frac{\gamma\delta_i^2}8\eqByRef{eq:cutdistdeltai}\eps_i\;,
\]
contradicting~\eqref{eq:cutnormdist}.
\end{claimproof}

We record an immediate consequence of Claim~\ref{clm:vanashRi}.
\begin{claim}\label{clm:wlHa}
For every $i\ge i_0$, property~\ref{en:size} holds for $R_i$, with $\delta_i$ in place of $\alpha$.    
\end{claim}
\begin{claimproof}
Apply Claim~\ref{clm:vanashRi} with $\delta:=\delta_i$; this is legitimate because $\delta_i\in[\delta_i/2,\beta_0]$ by~\eqref{eq:i0}.
\end{claimproof}
Since at least one of~\ref{en:size}--\ref{en:connected} fails for $R_i$ with $\delta_i$ in place of $\alpha$, while~\ref{en:size} holds by Claim~\ref{clm:wlHa}, we know that for every $i\ge i_0$ there exists a set $A_i\subseteq V(R^*_i)$ with $|A_i|\le C\gamma \delta_i n_i$ such that~\ref{en:halfcov} or~\ref{en:connected} does not hold for $R^*_i-A_i$. Set $R_i':=R_i^*-A_i$ and $n_i':=v(R_i')$. The graph $R_i'$ is obtained from $R_i$ by deleting the vertices of $\SMALLDEG_{R_i}(\delta_i)\cup A_i$, and there are
\begin{equation}\label{eq:deletedcount}
|\SMALLDEG_{R_i}(\delta_i)|+|A_i|\leBy{Claim~\ref{clm:vanashRi}}(1+C)\gamma \delta_i n_i\leBy{$\gamma<\frac{1}{20C}$}\tfrac{1}{10}\delta_i n_i
\end{equation}
of them. In particular,
\begin{equation}\label{eq:nprime}
\tfrac{9}{10}n_i\le n_i'\le n_i\;.
\end{equation}
Let $W_{R_i'}$ be a graphon representation of $R_i'$ given by Lemma~\ref{lem:delete}. Combining that lemma with~\eqref{eq:cutnormdist} and~\eqref{eq:deletedcount}, we have
\begin{equation}\label{eq:Ri*convergetoW}
\cutn{W_{R_i'}-W}<2(C+1)\gamma\delta_i+\eps_i<3C\delta_i\;.
\end{equation}
We first note that $R_i'$ has a guaranteed minimum weighted degree.
\begin{claim}\label{clm:Ri*mindeg}
For every $i\ge i_0$ and every $v\in V(R_i')$ we have $\wdeg_{R_i'}(v)\ge \delta_i n_i'/2$.
\end{claim}
\begin{claimproof}
Indeed, take an arbitrary vertex $v\in V(R_i')$. As $v\notin \SMALLDEG_{R_i}(\delta_i)$, we have $\wdeg_{R_i}(v)\ge \delta_in_i$. Deleting a vertex decreases the weighted degree of each remaining vertex by at most~1, and thus 
\[
\wdeg_{R_i'}(v)\ge \wdeg_{R_i}(v)-|\SMALLDEG_{R_i}(\delta_i)|-|A_i|\geByRef{eq:deletedcount} \delta_in_i-\tfrac{1}{10}\delta_i n_i\geByRef{eq:nprime} \delta_in_i'/2
\;.
\]
\end{claimproof}

Furthermore, $R_i'$ inherits the property that there are few vertices of small weighted degree.

\begin{claim}\label{clm:expanddegRi*}
For every $i\ge i_0$ and $\beta\in [0,\beta_0/2)$ we have $|\SMALLDEG_{R_i'}(\beta)|\le\beta n_i'/10$.
\end{claim}
\begin{claimproof}
If $\beta<\delta_i/2$, then $\SMALLDEG_{R_i'}(\beta)=\emptyset$ by Claim~\ref{clm:Ri*mindeg}. So assume that $\beta\ge\delta_i/2$. Deleting a vertex decreases the weighted degree of each remaining vertex by at most~1, and so every $v\in \SMALLDEG_{R_i'}(\beta)$ satisfies $\wdeg_{R_i}(v)\le \wdeg_{R_i'}(v)+|\SMALLDEG_{R_i}(\delta_i)|+|A_i|\le \beta n_i'+(1+C)\gamma\delta_i n_i$, which is at most $\left(\beta+(1+C)\gamma\delta_i\right) n_i$ by~\eqref{eq:nprime}. In other words,
\begin{equation}\label{eq:tos}
\SMALLDEG_{R_i'}(\beta)\subseteq \SMALLDEG_{R_i}\left(\beta+(1+C)\gamma\delta_i\right)\;.
\end{equation}
Recalling the assumption $\gamma\in(0,\tfrac{1}{20C})$ we get the bound $\beta+(1+C)\gamma\delta_i\le \beta+\frac{\delta_i}{10}\le 1.2\beta$. In particular, $1.2\beta\in[\delta_i/2,\beta_0]$, and hence Claim~\ref{clm:vanashRi} applies to $1.2\beta$. We continue with~\eqref{eq:tos},
\begin{align*}
|\SMALLDEG_{R_i'}(\beta)|&\le|\SMALLDEG_{R_i}(1.2\beta)|\lBy{Claim~\ref{clm:vanashRi}}1.2\gamma\beta n_i\leByRef{eq:nprime}\tfrac43\gamma\beta n_i'<\beta n_i'/10\;,
\end{align*}
as was needed.
\end{claimproof}

In Claim~\ref{clm:wlHa}, we have already established property~\ref{en:size}. The next claim establishes property~\ref{en:connected}, too.
\begin{claim}\label{clm:wlHc}
For every $i\ge i_0$, property~\ref{en:connected} holds for $R_i'=R^*_i-A_i$.
\end{claim}
\begin{claimproof}
Suppose for a contradiction that $R_i'$ is not connected. Let $Z_1\subseteq V(R_i')$ be a connected component of $R_i'$ of minimum size. Set $Z_2:=V(R_i')\setminus Z_1$, and $\Delta:=\max_{v\in Z_1}\wdeg_{R_i'}(v)$. Note that $|Z_1|\ge \Delta$: indeed, all the neighbours of a vertex $v\in Z_1$ lie in $Z_1$, and so $\wdeg_{R_i'}(v)\leByRef{eq:wdegdeg}\deg_{R_i'}(v)\le|Z_1|-1$.
We first consider the case that $\Delta \ge \tfrac{\beta_0}{2}n_i'$. In particular, $|Z_1|\ge \tfrac{\beta_0}{2}n_i'$. By~\eqref{eq:Ri*convergetoW} and~\eqref{eq:i0} we have that $\cutn{W_{R_i'}-W}<\kappa$. Therefore by Proposition~\ref{prop:connectedgraphon} we have $e_{R_i'}(Z_1,Z_2)>0$, contradicting that $Z_1$ is a connected component. Next, consider the case that $\Delta< \tfrac{\beta_0}{2}n_i'$. Observe that $Z_1\subseteq \SMALLDEG_{R_i'}(\Delta/n_i')$. Therefore,
$$\Delta \le|Z_1|\le |\SMALLDEG_{R_i'}(\Delta/n_i')|\leBy{Claim~\ref{clm:expanddegRi*}}\Delta/10\;,$$
a contradiction, as $\Delta>0$ by Claim~\ref{clm:Ri*mindeg}.
\end{claimproof}

By the choice of $A_i$, at least one of~\ref{en:halfcov} and~\ref{en:connected} fails for $R_i'$, while~\ref{en:connected} holds by Claim~\ref{clm:wlHc}. We conclude that~\ref{en:halfcov} does not hold. Hence for every $i\ge i_0$, the graph $R_i'$ is not uniquely half-covered, that is, it has a half-integral vertex cover $c_i:V(R_i')\to\{0,\frac12,1\}$ of total weight at most $\frac{n_i'}2$ which is not constant-$\frac12$. We set $S_i:=c_i^{-1}(0)$ and $P_i:=c_i^{-1}(1)$. From the bound on the total weight of $c_i$, we have $\frac12|c_i^{-1}(\frac12)|+|c_i^{-1}(1)|\le \frac{n_i'}2=\frac12|c_i^{-1}(0)|+\frac12|c_i^{-1}(\frac12)|+\frac12|c_i^{-1}(1)|$, or equivalently, 
\begin{equation}\label{eq:SH}
|S_i|\ge |P_i|\;.
\end{equation}
Also, since $c_i$ is not constant-$\frac12$, the set $S_i\cup P_i$ is nonempty, and thus~\eqref{eq:SH} gives that $S_i$ is nonempty. We may hence set $\Delta_i:=\max_{v\in S_i}\wdeg_{R_i'}(v)$ and take a vertex $v_i\in S_i$ whose weighted degree is $\Delta_i$. Since $c_i(v_i)=0$ and $c_i$ is a fractional vertex cover, we see that each neighbour of $v_i$ is in $P_i$. In particular,
\begin{equation}\label{eq:PiBig}
\Delta_i=\wdeg_{R_i'}(v_i)\leByRef{eq:wdegdeg}\deg_{R_i'}(v_i)\le|P_i|\;.
\end{equation}

To conclude the proof, we distinguish two exhaustive, though not necessarily disjoint, cases.

\emph{Case~I. For infinitely many $i$ we have $\Delta_i\ge \beta_0n_i'/2$.}
We write $f_i:\Omega\to\{0,\frac12,1\}$ for the half-integral vertex cover of $W_{R_i'}$ which corresponds to $c_i$, that is, $f_i$ is constant-$c_i(v)$ on the part of measure $\frac1{n_i'}$ representing the vertex $v\in V(R_i')$. Note that $\|f_i\|_1=\frac1{n_i'}\sum_{v\in V(R_i')}c_i(v)\le \frac12$. 
We apply Proposition~\ref{prop:weakstarlimit} to the sequence of graphons $W_{R_i'}$ that satisfy Case~I, together with their half-integral vertex covers $f_i$. These graphons converge to $W$ by~\eqref{eq:Ri*convergetoW}. The assumption of Case~I gives us that 
\[
\mu(f_i^{-1}(0))=\frac{|S_i|}{n_i'}\geBy{\eqref{eq:SH},\eqref{eq:PiBig}}\frac{\Delta_i}{n_i'}\geBy{Case~I} \beta_0/2\;.
\]
So, Proposition~\ref{prop:weakstarlimit} provides a half-integral vertex cover $f$ of $W$ with $\mu(f^{-1}(0))\ge\frac{\beta_0}2$ by~\ref{it:wslim-zeroset}, and with $\|f\|_1\le\frac12$ by~\ref{it:wslim-norm}. In particular, $f$ is not equal to $\frac12$ almost everywhere. By Fact~\ref{fct:halfintvcispeninsula}, the graphon $W$ has a peninsula, which contradicts our key assumption~\ref{en:mainHC3}. 

\emph{Case~II. There exists at least one $i\ge i_0$ for which we have $\Delta_i<\beta_0 n_i'/2$.}
Take such an index $i$. Observe that $S_i\subseteq \SMALLDEG_{R_i'}(\Delta_i/n_i')$. From the assumption of Case~II, Claim~\ref{clm:expanddegRi*} applies to $\SMALLDEG_{R_i'}(\Delta_i/n_i')$. Hence,
$$
\Delta_i\leByRef{eq:PiBig}|P_i|\leByRef{eq:SH} |S_i|\le |\SMALLDEG_{R_i'}(\Delta_i/n_i')| \leBy{Claim~\ref{clm:expanddegRi*}} \Delta_i/10\;,
$$
a contradiction, as $\Delta_i>0$ by Claim~\ref{clm:Ri*mindeg}.
\end{proof}

\section{Absorption}\label{sec:Absoption}
In order to convert the almost spanning cycle---obtained via the regularity lemma and Łuczak's connected matchings method---into a true Hamilton cycle, we employ the absorption method, originally introduced by Rödl, Ruciński, and Szemerédi~\cite{RoRuSzAbsor06}. The absorber needed for our purposes is constructed in Lemma~\ref{lem:abs}.

We say that a graph~$G$ is \emph{$(\zeta,X, U, \ell)$-connected} if for every pair of distinct vertices $u,v \in X$ there are at least $\zeta |U|^\ell$ paths from $u$ to $v$ which have exactly $\ell$ internal vertices, all of them in~$U$.

\begin{defi}\label{def:absorber}
    Given $\gamma>0$, an $n$-vertex graph~$G$, and a set~$U\subseteq V(G)$, we say that a path~$\cQ\subseteq G$ is a~\emph{$(U,\gamma)$-absorbing path} if for every subset of vertices~$B\subseteq U$ of size at most~$\gamma n$, there is a path in~$G$ with vertex set~$V(\cQ)\cup B$ and with the same ends as~$\cQ$.
\end{defi}

\begin{lem}\label{lem:abs}
    Given~$\eps, \zeta>0$ and~$\ell, t\in \N$ there is a~$\gamma'>0$ such that for every~$\gamma\in (0,\gamma')$, the following holds for every sufficiently large~$n\in \N$.
    Let~$G$ be a graph on~$n$ vertices and~$\emptyset\neq U\subseteq V(G)$ such that~$G$ is~$(\zeta, U , U, t)$-connected. 
    Suppose that for every~$v\in U$ there are at least~$\eps n^{2\ell}$ copies of~$C_{2\ell+1}$ containing~$v$ whose remaining~$2\ell$ vertices lie in~$U$. 
    Then~$G$ contains a~$(U,\gamma^2)$-absorbing path with at most~$\gamma n$ vertices, all of which lie in~$U$.
\end{lem}

\begin{proof}
In the proof we use two auxiliary constants, $c\in\N$ and $\epsabs>0$, which are introduced in Claim~\ref{cl:ErdosDouble} and in its proof, and which depend on $\eps$ and $\ell$ only. Since $\gamma'$ may depend on $\eps,\zeta,\ell$ and $t$, we may assume throughout the proof the constant hierarchy
\begin{equation}\label{eq:conshier}
\tfrac{1}{n}\ll\gamma\ll\epsabs,\eps,\zeta,\tfrac{1}{\ell},\tfrac{1}{t}\;.
\end{equation}
Let~$C_{2\ell+1}^*(2)$ be the blow-up of~$C_{2\ell+1}$ on an auxiliary vertex set, where every vertex is duplicated once except for one vertex. 
More precisely,~$C_{2\ell+1}^*(2)$ is a graph with vertex set~$\{a_0, a_1,b_1, \dots, a_{2\ell}, b_{2\ell}\}$ and all edges of the form~$a_ia_{i+1}, b_ib_{i+1}, a_ib_{i+1},b_ia_{i+1}$ for every~$i\in [2\ell-1]$ plus the edges~$a_0a_1$, $a_0b_1$,~$a_0a_{2\ell}$, and~$a_0b_{2\ell}$.

The following claim follows from a well-known result of Erd\H os. 
\begin{claim}\label{cl:ErdosDouble}
For every vertex $v\in U$ there are at least $\epsabs n^{4\ell}$ sets $S\subseteq U\setminus\{v\}$ of size $4\ell$ such that $S\cup\{v\}$ spans a copy of $C_{2\ell+1}^*(2)$ with~$v$ playing the role of~$a_0$. (For this $\epsabs$ depends on $\eps$ and $\ell$ only.)
\end{claim}
\begin{claimproof}
Let~$c=c(\ell)\in\N$ be sufficiently large for the application of the (partite) Ramsey theorem below; in particular, $c$ depends on~$\ell$ only.
Take an arbitrary vertex~$v\in U$ and construct an auxiliary~$2\ell$-uniform hypergraph~$\cC$ with vertex set~$U\setminus \{v\}$ and edges given by
\begin{equation*}
    E(\cC) = \{e\subseteq U\setminus \{v\} \colon e\cup \{v\} \text{ spans a copy of~$C_{2\ell+1}$}\}\,. 
\end{equation*}
Since at most~$\tfrac{(2\ell)!}{2}$ copies of~$C_{2\ell+1}$ containing~$v$ have the same vertex set, the assumption of the lemma gives $|E(\cC)|\ge \eps' n^{2\ell}$, where $\eps':=\tfrac{2\eps}{(2\ell)!}$.
Write $N:=|U\setminus\{v\}|$ for the number of vertices of $\cC$; since $|E(\cC)|\le N^{2\ell}$, we have $N\ge (\eps')^{1/(2\ell)}n$. If~$K_{c,\dots, c}^{(2\ell)}$ denotes the complete~$2\ell$-partite~$2\ell$-uniform hypergraph with~$c$ vertices in each partite class, a classical result of Erd\H os~\cite[Corollary 2]{Erdos83supersat}, applied to $\cC$ (which has at least $\eps' n^{2\ell}\ge \eps' N^{2\ell}$ edges), provides a constant $\eta_0=\eta_0(\eps',c,\ell)>0$ such that $\cC$ contains at least $\eta_0 N^{2c\ell}$ many copies of $K_{c,\dots, c}^{(2\ell)}$. Setting $\eta:=\eta_0\cdot(\eps')^{c}$ and using the lower bound on $N$, we get
\begin{equation}\label{eq:Kc}
    \text{$\cC$ contains at least $\eta n^{2c\ell}$ many copies of } K_{c,\dots, c}^{(2\ell)}\,.
\end{equation}
Each edge of such a copy of~$K_{c,\dots, c}^{(2\ell)}$ in~$\cC$, together with~$v$, represents a copy of~$C_{2\ell+1}$ in~$G$ in one of the~$(2\ell)!$ possible cyclic orderings (with respect to the vertex partition). 
Since~$c$ is large compared to~$\ell$, an application of the (partite) Ramsey theorem with~$(2\ell)!$ many colours entails that every copy of $K_{c,\dots, c}^{(2\ell)}$ in~$\cC$ contains a copy of~$K_{2,\dots, 2}^{(2\ell)}$ in which all edges together with~$v$ represent copies of~$C_{2\ell+1}$ in~$G$ with the same cyclical ordering. 
We denote by~$\Kcyc$ a copy of~$K_{2,\dots, 2}^{(2\ell)}$ in~$\cC$ satisfying this property. 
The argument above entails
\begin{equation}\label{eq:Kcyc}
    \text{each copy of $K_{c,\dots, c}^{(2\ell)}$ in~$\cC$ contains a copy of~$\Kcyc$}\,.
\end{equation}

Let~$\kappa^{\text{cyc}}$ denote the number of copies of~$K_{2,\dots, 2}^{(2\ell)}$ in~$\cC$ with the above property, and let~$\kappa_c$ denote the number of copies of~$K_{c,\dots, c}^{(2\ell)}$ in~$\cC$. 
Observe that each such copy of $K_{2,\dots, 2}^{(2\ell)}$ can be contained in at most~$n^{2\ell(c-2)}$ many copies of~$K_{c,\dots, c}^{(2\ell)}$. 
Combining this observation with \eqref{eq:Kc} and \eqref{eq:Kcyc}, we get
\begin{align*}
    \eta\cdot n^{2c\ell} \le \kappa_{c} \le \kappa^{\text{cyc}}n^{2\ell(c-2)}\,,
\end{align*}
which implies~$\kappa^{\text{cyc}}\ge \eta n^{4\ell}$.

Each~$\Kcyc$ yields a set $S$ as in the claim, namely the set of its $4\ell$ vertices. Indeed, let $\{a_i,b_i\}$, $i\in[2\ell]$, be the classes of a given~$\Kcyc$, indexed according to the common cyclic ordering. For every $i\in[2\ell-1]$ and every choice of $x\in\{a_i,b_i\}$ and $y\in\{a_{i+1},b_{i+1}\}$, some edge of the~$\Kcyc$ contains both $x$ and $y$, and the corresponding copy of~$C_{2\ell+1}$ contains the edge $xy$; similarly, $v$ is adjacent to each of $a_1,b_1,a_{2\ell}$ and $b_{2\ell}$. Hence $S\cup\{v\}$ spans a copy of~$C_{2\ell+1}^*(2)$ with $v$ in the role of $a_0$. As each such set $S$ arises in this way from at most~$(4\ell)!$ many~$\Kcyc$, setting $\epsabs:=\tfrac{\eta}{(4\ell)!}$ proves the claim; note that $\epsabs$ depends on $\eps$ and $\ell$ only, as $c$ does.
\end{claimproof}

Consider a copy of~$C_{2\ell+1}^*(2)$ in $G$. 
We hence view its vertices $\{a_0, a_1,b_1, \dots, a_{2\ell}, b_{2\ell}\}$ as in $G$. 
It is easy to see that the following two vertex orderings form paths:
 \begin{enumerate}[label=(V\arabic*)]
     \item \label{it:nov} $a_1,a_2,b_1,b_2,\dots, a_{2i-1}, a_{2i},b_{2i-1},b_{2i}, \dots, a_{2\ell-1}, a_{2\ell}, b_{2\ell-1}, b_{2\ell}$, and 
     \item \label{it:v} 
     $a_1,a_0, a_{2\ell}, a_{2\ell-1}, \dots, a_2, b_1, b_2, \dots, b_{2\ell}$
 \end{enumerate}
Note that both paths start and end in the same vertices and both span the same set of vertices except for~$a_0=v$, which is only contained in the second path. 
Given a vertex~$v\in U$, we say that an injective map~$C$ from~$\{a_1,b_1,\ldots,a_{2\ell},b_{2\ell}\}$ to~$U\setminus\{v\}$ is an \textit{absorber for~$v$} if the extension of~$C$ by~$a_0\mapsto v$ maps every edge of~$C_{2\ell+1}^*(2)$ to an edge of~$G$, that is, if it gives a copy of~$C_{2\ell+1}^*(2)$ in~$G$. We write~$V(C)\subseteq U$ for the image of~$C$, a set of~$4\ell$ vertices. Let us stress that an absorber carries a labelling of these $4\ell$ vertices: this labelling specifies the path~\ref{it:nov} on~$V(C)$ and the path~\ref{it:v} on~$V(C)\cup\{v\}$, and both these paths have the same two ends~$C(a_1)$ and~$C(b_{2\ell})$, irrespective of~$v$. This is essential for us in two ways. First, one and the same map~$C$ is typically an absorber for many vertices~$v$; since the ends of~\ref{it:nov} and of~\ref{it:v} do not depend on~$v$, such a~$C$ can be used to absorb any one of them, and it is only the greedy assignment described below that decides which one it will be. Second, one and the same set of~$4\ell$ vertices may carry several different labellings, and thus give paths~\ref{it:nov} with different ends; we must therefore make sure that at most one labelling of each such set is used.

For every vertex~$v\in U$ define $\cA_{v}= \{C\colon C \text{ is an absorber for~$v$}\}$. Further, let $\cA = \bigcup_{v\in U} \cA_v$. Since each of the sets provided by Claim~\ref{cl:ErdosDouble} gives rise to at least one absorber for~$v$, we have
\begin{equation}
\label{eq:cardinalitycAv}
|\cA_v|\ge \epsabs n^{4\ell}\text{ for every $v\in U$\;.}
\end{equation}
By assumption, $U$ is nonempty. Fix a vertex $v\in U$ and delete $v$ from each of the at least $\eps n^{2\ell}$ cycles $C_{2\ell+1}$ guaranteed for $v$. This gives at least $\eps n^{2\ell}$ distinct paths on $2\ell$ vertices, all of them lying in $U$. As there are at most $|U|^{2\ell}$ such paths, we get $|U|^{2\ell}\ge\eps n^{2\ell}$, that is,
\begin{equation}
\label{eq:sizeU}
|U|\ge \sqrt[2\ell]{\eps}n\;.
\end{equation}
Further, an absorber is an injective map defined on a set of $4\ell$ symbols, and so there are at most $n^{4\ell}$ absorbers in total. Hence, using~\eqref{eq:cardinalitycAv} for the lower bound, we get
\begin{equation}
\label{eq:cardinalitycA}
n^{4\ell}\ge|\cA|\ge |\cA_v|\ge \epsabs n^{4\ell}\;.
\end{equation}

\begin{claim}\label{cl:randomsubset}
There exists a set $\hat\cA\subseteq \cA$ such that 
\begin{enumerate}[label=(\alph*)]
    \item\label{it:abs-size} $|V(\hat\cA)| \le \tfrac{\gamma n}{2t}$,
    \item\label{it:abs-disjoint} $V(C_1)\cap V(C_2) = \emptyset$ for every pair~$C_1, C_2\in \hat\cA$ with~$C_1\neq C_2$, and
    \item\label{it:abs-many} for every vertex~$v\in U$ we have~$|\hat \cA\cap \cA_v|\ge \gamma^2 n$.
\end{enumerate}
\end{claim}
\begin{claimproof}
Since~$\gamma\ll \epsabs,\tfrac1\ell$ by~\eqref{eq:conshier}, we have~$\sqrt{\gamma}\le\tfrac{\epsabs^2}{20\ell}$, which is equivalent to~$\tfrac{4\gamma^2}{\epsabs}\le\tfrac{\epsabs\gamma^{3/2}}{5\ell}$. Further,~$\epsabs\le 1$ by~\eqref{eq:cardinalitycA}, and so the latter quantity is at most~1. Hence we may choose a constant~$\alpha \in (0,1]$ such that
\begin{equation}\label{eq:amimrtv}
\frac{4\gamma^2}{\epsabs} \le \alpha \le  \frac{\epsabs\gamma^{3/2}}{5\ell}\,.    
\end{equation}
For any $\cA_*\subseteq \cA$, let $\cB(\cA_*)=\{(C_1,C_2) \in \cA_*^2\colon C_1\neq C_2 \textrm{ and } V(C_1)\cap V(C_2) \neq \emptyset \}$. We want to get an upper bound on $|\cB(\cA)|$. 
Let us count those pairs $(C_1,C_2)\in \cA^2$ for which $C_1\neq C_2$ and $V(C_1)\cap V(C_2)\ni w$ for a given vertex $w\in V(G)$. In such a situation, $C_1$ is determined by the symbol which it maps to $w$ and by its remaining $4\ell-1$ vertices, which gives at most $4\ell\cdot n^{4\ell-1}$ choices; the same applies to $C_2$. Summing over all $w\in V(G)$, we get
\begin{equation}\label{eq:boundBA}
 | \cB(\cA)|\le n\cdot \left(4\ell\cdot n^{4\ell-1}\right)^2= 16\ell^2n^{8\ell-1}\;.   
\end{equation}

Let~$\hat\cA_0$ be a random collection picked from~$\cA$, by including each element independently with probability~$p = \alpha n/|\cA|$.
We have $\Expectation[|V(\hat\cA_0)|] \le 4\ell\cdot |\cA|\cdot p= 4\ell\cdot \alpha n\le \tfrac{\gamma n}{4t}$, where the last inequality follows from~\eqref{eq:amimrtv} and~\eqref{eq:conshier}. By Markov's inequality,
\begin{equation}\label{eq:errprob1}
\Probability\left[|V(\hat\cA_0)|\ge \tfrac{\gamma n}{2t}\right] \le 1/2\;.
\end{equation}

Next, we look at $\Expectation[|\cB(\hat\cA_0)|]$. We have 
\[
\Expectation[|\cB(\hat\cA_0)|]=p^2\cdot |\cB(\cA)|
\leBy{\eqref{eq:cardinalitycA}, \eqref{eq:boundBA}}
\left(\frac{\alpha n}{\epsabs n^{4\ell}}\right)^2\cdot
16\ell^2n^{8\ell-1}
\leByRef{eq:amimrtv} \tfrac{16}{25}\gamma^3 n\le\gamma^3 n\;.
\]
By Markov's inequality,
\begin{equation}\label{eq:errprob2}
\Probability[|\cB(\hat\cA_0)|\ge 4\gamma^3n] \le 1/4\;.
\end{equation}

Last, we look at $\Expectation[|\hat\cA_0\cap \cA_v|]$ for any vertex~$v\in U$. We have 
$\Expectation[|\hat\cA_0\cap \cA_v|]=p\cdot |\cA_v| =\alpha n\cdot \frac{|\cA_v|}{|\cA|}$. We can now use~\eqref{eq:cardinalitycA}, \eqref{eq:cardinalitycAv}, and \eqref{eq:amimrtv} and get $\Expectation[|\hat\cA_0\cap \cA_v|]\ge 4\gamma^2 n$. Therefore, by the Chernoff inequality (Lemma~\ref{lem:Chernoff}\ref{it:Chernoff1}) and the union bound we obtain 
\begin{equation}\label{eq:errprob3}
\Probability\big[\text{there is a vertex~$v\in U$ with }|\hat\cA_0\cap \cA_v| \le 2\gamma^2 n\big] \le n\exp(-\gamma^2 n/2) < 1/4\,.
\end{equation}
Thus, with positive probability none of the events in~\eqref{eq:errprob1}, \eqref{eq:errprob2}, and~\eqref{eq:errprob3} occurs. Let us fix such an outcome~$\hat \cA_0$.

Finally, we obtain~$\hat\cA$ from $\hat\cA_0$ by removing all absorbers that are contained in some pair in~$\cB(\hat\cA_{0})$. Since the event in~\eqref{eq:errprob1} does not occur, we have~\ref{it:abs-size}. The last removal step ensures~\ref{it:abs-disjoint}. It remains to verify~\ref{it:abs-many}. Each absorber removed in the last step is the first coordinate of at least one pair in~$\cB(\hat\cA_0)$, and distinct removed absorbers give rise to distinct such pairs; hence the number of removed absorbers is at most $|\cB(\hat\cA_0)|<4\gamma^3n$, where we used that the event in~\eqref{eq:errprob2} does not occur. Combining this with the lower bound $|\hat\cA_0\cap \cA_v| > 2\gamma^2 n$, which we have as the event in~\eqref{eq:errprob3} does not occur, we get
$$|\hat\cA\cap \cA_v|> 2\gamma^2 n-4\gamma^3n\ge \gamma^2n$$
for every~$v\in U$, as needed.
\end{claimproof}

Observe that the labelling carried by an absorber~$C\in \cA_v$ determines a path on the vertex set~$V(C)$ as in~\ref{it:nov}, i.e.\ a path not containing~$v$. 
By~\ref{it:abs-disjoint} these paths are pairwise vertex-disjoint for the members of~$\hat\cA$; in particular, no set of~$4\ell$ vertices carries two different labellings in~$\hat\cA$, and so~$\hat \cA$ can be viewed as a disjoint collection of paths. 
Moreover, given a set of vertices~$U'\subseteq U\setminus V(\hat\cA)$ of size at most~$\gamma^2 n$, we can greedily assign each vertex~$v\in U'$ to its own absorber~$C\in \hat\cA \cap \cA_v$; this is possible since at each step fewer than $|U'|\le \gamma^2n\le |\hat\cA\cap \cA_v|$ absorbers have been used already, where the last inequality is Claim~\ref{cl:randomsubset}\ref{it:abs-many}.
In this way, using the paths from~\ref{it:v}, the family~$\hat\cA$ together with the vertices in~$U'$ can also be viewed as a collection of paths, with the same endings as the paths in~$\hat\cA$ and with vertex set~$V(\hat\cA)\cup U'$. 

We only need to connect the paths in~$\hat\cA$. 
We do this directly by fixing an ordering~$\hat \cA=\{C_1,\dots , C_{r}\}$ and iteratively connecting each~$C_i$ with the next~$C_{i+1}$. 
The ends of these paths lie in~$U$, and so in each case we have at least~$\zeta |U|^t$ potential connections with~$t$ internal vertices from $U$.
By~\ref{it:abs-size} and~\ref{it:abs-disjoint} we have $r=\tfrac{|V(\hat\cA)|}{4\ell}\le \tfrac{\gamma n}{8\ell t}$. Hence at most $r\cdot t\le \tfrac{\gamma n}{8\ell}$ many vertices are blocked from previous connections, and at most $|V(\hat\cA)|\le \tfrac{\gamma n}{2t}$ by vertices of~$\hat \cA$, which is $b\le \gamma n$ many vertices in total. A potential connection which uses a given vertex is determined by the position of that vertex among the $t$ internal ones and by the remaining $t-1$ internal vertices. Hence at any step at most $t\cdot b\cdot|U|^{t-1}$ of the potential connections are blocked, and since
\[
t\cdot b\le \gamma tn\lByRef{eq:conshier}\zeta\sqrt[2\ell]{\eps}n\leByRef{eq:sizeU}\zeta|U|\;,
\]
at every step some connection is available, and we join all the paths into a single path~$\cQ$, the internal vertices of the connections being taken from~$U$. In particular, $V(\cQ)\subseteq U$ and $|V(\cQ)|\le |V(\hat\cA)|+r\cdot t\le \gamma n$. 

Finally, $\cQ$ is $(U,\gamma^2)$-absorbing. Indeed, let~$B\subseteq U$ with~$|B|\le \gamma^2n$ be given and apply the above to~$U':=B\setminus V(\cQ)$. For each~$v\in U'$, let~$C\in \hat\cA\cap \cA_v$ be the absorber assigned to~$v$; the subpath of~$\cQ$ spanned by~$V(C)$ is the path~\ref{it:nov} of that very absorber, and we replace it by the path~\ref{it:v} of the same absorber, which spans~$V(C)\cup\{v\}$ and has the same two ends. Performing these replacements for all~$v\in U'$, we obtain a path with vertex set~$V(\cQ)\cup B$ and with the same ends as~$\cQ$. 
\end{proof}

\section{Proof of the main part of Theorem~\ref{thm:mainHC}}\label{sec:proof}
\stepcounter{proofctr}\setcounter{claim}{0}

Suppose that $W$ satisfies conditions \ref{en:mainHC1}--\ref{en:mainHC3}. Since the assertion of Theorem~\ref{thm:mainHC} is asymptotic, it suffices to prove that for every $\eta>0$ we have $\Probability[\G(n,W)\text{ is Hamiltonian}]\ge 1-\eta$ for every sufficiently large~$n$. So, let $\eta>0$ be given and suppose that $n$ is sufficiently large. For the rest of the proof, we introduce additional constants with the following hierarchy:
\begin{equation}\label{eq:mainhierarchy}
0<\tfrac{1}{n} \ll \tfrac{1}{r_2}<\eps_2 \ll d_2\ll \rho \ll\gamma  \ll \zeta \ll \tfrac{1}{r_1} < \eps_1 \ll d_1 \ll \alpha \ll \eps_0\ll \beta\ll 1\;.
\end{equation}
Here $r_1$ and $r_2$ are not chosen by us: they are the numbers of clusters returned by Theorem~\ref{thm:reglem} when applied with $\eps_1$ and with $\eps_2$, respectively. Theorem~\ref{thm:reglem} guarantees that $\tfrac{1}{r_1}<\eps_1$ and $\tfrac1{r_2}<\eps_2$, which is the reason for the two occurrences of `$<$' in~\eqref{eq:mainhierarchy}, and that $r_1\le T_1$ and $r_2\le T_2$, where $T_1$ and $T_2$ are constants depending only on $\eps_1$ and on $\eps_2$, respectively. The constants standing to the left of $\tfrac1{r_1}$ in~\eqref{eq:mainhierarchy} are accordingly chosen small in terms of $T_1$ rather than of $r_1$; as $r_1\le T_1$, they are then small in terms of $r_1$ as well. The same applies to $\tfrac{1}{r_2}$, and to the choice of~$n$.
We also require that $\eps_0\le\eta/2$. Let us stress that $\eps_0$ is chosen before $\alpha$ in this hierarchy. Indeed, we shall apply Proposition~\ref{prop:lowdegreePaths} with $\eps_0$ in place of $\eps$, and we require in addition that $4\alpha<\alpha_0$, where $\alpha_0=\alpha_0(\eps_0)$ is the constant provided by that proposition; this is legitimate as $\alpha_0$ depends on $W$ and on $\eps_0$ only. Likewise, we require that $\alpha\le\alpha_0'$, where $\alpha_0'$ is the constant provided by Proposition~\ref{prop:clustergraph} applied with $\beta$ in place of $\gamma$ and with $C:=\tfrac{1}{100\beta}$; this is legitimate as $\alpha_0'$ depends on $W$ and on $\beta$ only.

Observe that Proposition~\ref{prop:lowdegreePaths} (applied with $2\alpha$ in place of $\alpha$ and with $\eps_0$ in place of $\eps$) together with Lemma~\ref{lem:sampledist} gives that the random graph $G\sim\G(n,W)$ satisfies all of the following properties with probability at least $1-2\eps_0$.
    \begin{enumerate}[label=(F\arabic*)]
        \item\label{it:proplowdeg} $G$ contains a collection of vertex-disjoint paths~$\cP$ such that 
        \begin{enumerate}[label=(\roman*)]
        \item \label{it:main1-lowdeginP} $\{v\in V\colon \deg_G(v)<2\alpha n\}\subseteq V(\cP)$,
        \item \label{it:main1-highenddeg} the end vertices of each path have degree at least $2\alpha n$,
        \item \label{it:boundvrtP} $|V(\cP)|<2\eps_0\alpha n$, and 
        \item \label{it:main1-fewpaths} $|\cP|<2/\alpha$.
        \end{enumerate}
        \item\label{it:cutnorm} There exists a graphon representation $W_G$ of $G$ with $\cutn{W_G-W}\le \eps_2/2$.
    \end{enumerate}
    We proceed to show that then~$G=(V,E)$ contains a Hamilton cycle. As the properties above hold with probability at least $1-2\eps_0\ge 1-\eta$, this will prove the theorem.

   \subsubsection*{Step 1: Covering all low degree vertices with few paths}
    We fix $\cP$ given by~\ref{it:proplowdeg} and note that in particular $\cP$ contains all vertices of degree less than~$\alpha n$. 
    Set~$X:= \{ v\in V \colon \deg_G(v)\ge \alpha n\}$ and note that then~$V = X\cup V(\cP)$.  

    \subsubsection*{Step 2: Finding an absorber}
    In this step, we want to use Lemma~\ref{lem:abs} to find an~$(X\setminus V(\cP), \gamma^2)$-absorbing path~$\cQ$.
    We apply Theorem~\ref{thm:reglem} with $\eps_1$ and with the trivial prepartition $V=\emptyset\sqcup V$, and obtain an~$\eps_1$-regular partition~$V_1,\dots, V_{r_1}$ of~$G$; recall that~$\tfrac1{r_1}<\eps_1$. 
    To use Lemma~\ref{lem:abs} with~$X\setminus V(\cP)$ in the role of~$U$, we need to find an odd~$L\le 2r_1$ such that~$G$ is~$(\zeta, X\setminus V(\cP),X\setminus V(\cP), L+1)$-connected and such that every vertex of~$X$ is contained in at least~$\zeta n^{L+1}$ cycles of length~$L+2$ whose remaining~$L+1$ vertices lie in $X\setminus V(\cP)$.
    Note that~$X\setminus V(\cP)=V\setminus V(\cP)$ is nonempty by~\ref{it:proplowdeg}\ref{it:boundvrtP}, as Lemma~\ref{lem:abs} requires.
    Let $R$ be the~$(\eps_1,d_1)$-reduced graph of $G$ with respect to the partition $\{V_i\}_{i\in[r_1]}$.
    By~\ref{it:cutnorm} we have $\cutn{W_G-W}\le\eps_2/2$.    
    By Fact~\ref{fact:distredgraph} there exists a graphon representation $W_{R}$ such that $\cutn{W_{R}-W_{G}}\le d_1+10\eps_1$. Altogether this gives~$\cutn{W-W_R}\le \cutn{W_{R}-W_{G}}+\cutn{W_G-W}\le\eps_2/2+d_1+10\eps_1 \le 2d_1<\tfrac{\beta\alpha^2}{8}$.
    With our choice of constants, Proposition~\ref{prop:clustergraph} applied with~$\beta$ instead of~$\gamma$ and with~$C=\tfrac{1}{100\beta}$ entails that
    \begin{equation}\label{eq:smalldegalpha}
    |\SMALLDEG_R(\alpha)|\le\alpha\beta r_1
    \end{equation}
    and that the subgraph~$R^* = R-\SMALLDEG_R(\alpha)$ satisfies~\ref{en:halfcov} and ~\ref{en:connected} for every~$A\subseteq V(R^*)$ with~$|A|\le \alpha r_1/100$. In the claim below, we show that this condition holds for a particular choice of $A$.        

    \begin{claim}\label{clm:GoodNeighbour}
        Set~$A := \{ i\in V(R^*)\colon |V_i\cap V(\cP)| \ge \beta |V_i|\}$.
        Then the set~$A$ satisfies properties~\ref{en:halfcov} and ~\ref{en:connected} from Proposition~\ref{prop:clustergraph}.
        Moreover, for every vertex~$v\in X$ there exists an index~$i_v\in V(R^*)\setminus A$ such that~
        $$|(N_G(v) \cap V_{i_{v}})\setminus V(\cP)| \ge \frac{\alpha}{2}|V_{i_v}|\;.$$
    \end{claim}
    \begin{claimproof}
    Observe that $|V_i|\ge \frac{n}{2r_1}$ for every $i$ by~\ref{R:1}, and so we have~$|A| \le \frac{2|V(\cP)|}{\beta n}r_1 \leBy{\ref{it:proplowdeg}\ref{it:boundvrtP}} \frac{4\eps_0\alpha r_1}{\beta}<\frac{\alpha r_1}{100}$. Therefore, Proposition~\ref{prop:clustergraph} applies.
    For the moreover part of the claim, let~$v\in X$ and note that by also using~\ref{it:proplowdeg}\ref{it:boundvrtP} and~\eqref{eq:smalldegalpha} we have
    \begin{align*}
        \Big|
        \Big(
        N_G(v)\cap \bigcup_{i\in V(R^*)\setminus A} V_i
        \Big)
        \setminus V(\cP)
        \Big|
        &\ge 
        \deg_G(v) - |V(\cP)| - \frac{2n}{r_1}\Big(|\SMALLDEG_R(\alpha)| + |A|\Big)\\
        &\ge \alpha n - 2\eps_0\alpha n - 2\Big(\alpha\beta + \frac{4\eps_0\alpha}{\beta}\Big)n
        \geBy{\eqref{eq:mainhierarchy}} 
        \frac{\alpha n}{2}\,.
    \end{align*}
    Thus, comparing the above bound with $\sum_{i\in V(R^*)\setminus A}|V_i|\le n$, we find an index~$i_v\in V(R^*)\setminus A$ with $|(N_G(v)\cap V_{i_v})\setminus V(\cP)|\ge\tfrac{\alpha}{2}|V_{i_v}|$, as needed.
    \end{claimproof}

    For every~$i\in V(R^*)\setminus A$ set~$V_i' := V_i\setminus V(\cP)$ and
    $V' := \bigcup_{i\in V(R^*)\setminus A} V_i'$.
    Notice that~$|V_i'|\ge (1-\beta)|V_i|$ for every~$i\in V(R^*)\setminus A$, and so, by~\ref{R:1}, the sets~$V_i'$ have sizes within a factor of~$1\pm2\beta$ of one another; the same computation as in the proof of Claim~\ref{clm:GoodNeighbour} also gives $|V'|\ge n-|V(\cP)|-\tfrac{2n}{r_1}(|\SMALLDEG_R(\alpha)|+|A|)\ge \tfrac n2$. Let~$\widetilde R$ be the~$(2\eps_1,d_1/2)$-reduced graph of~$G[V']$ with respect to the partition~$\{V_i'\}_{i\in V(R^*)\setminus A}$; that is,~$V(\widetilde R) := V(R^*)\setminus A$ and~$ij\in E(\widetilde R)$ if and only if~$(V_i',V_j')$ is a~$2\eps_1$-regular pair with density at least~$d_1/2$. Let us stress that we do not claim that this partition satisfies~\ref{R:1} with the parameter~$2\eps_1$; as~$\eps_1\ll\beta$, it need not, and we shall not need it either. 
    By Fact~\ref{fact:regpairinherit}, applied with~$1-\beta$ in place of~$\alpha$, we have that if a pair~$(V_i,V_j)$ is~$\eps_1$-regular with density at least~$d_1$, then~$(V_i',V_j')$ is~$2\eps_1$-regular with density at least~$d_1/2$.
    In other words,~$R^*-A$ is a subgraph of~$\widetilde R$. 
    
    We are now ready to show the properties needed for applying Lemma~\ref{lem:abs}.
    Due to~\ref{en:connected}, the graph $R^*-A$ is connected. 
    Further, due to~$\ref{en:halfcov}$ and Lemma~\ref{lem:OurCombOpt}\ref{en:OurCombOptNonBip}, the graph $R^*-A$ is not bipartite. 
    Since~$R^*-A$ is a spanning subgraph of~$\widetilde R$, the graph~$\widetilde R$ is connected and not bipartite as well.
    Therefore, we may apply Lemma~\ref{lem:path-oddcycle} to~$G[V']$ and its reduced graph~$\widetilde R$, with~$\zeta^*:=2^{2r_1+3}\zeta$ in place of~$\zeta$ and with~$\sigma:=\tfrac1{2r_1}$; the latter is legitimate because the sets~$V_i'$ have sizes within a factor of~$1\pm2\beta$ of one another and there are at most~$r_1$ of them, so each has at least~$\tfrac{1-2\beta}{1+2\beta}\cdot\tfrac{|V'|}{r_1}\ge\tfrac{|V'|}{2r_1}$ vertices, and the former is legitimate as~$\zeta^*\ll\tfrac1{r_1}$ by~\eqref{eq:mainhierarchy}. We obtain an odd integer~$L\le 2r_1$ such that between every two sets of vertices~$Y\subseteq V_i'$ with~$|Y|\ge d_1|V_i'|/4$ and~$Y'\subseteq V_j'$ with~$|Y'|\ge d_1|V_j'|/4$ there are at least
    \[
    \zeta^*|V'|^{L+1}\ge 2^{2r_1+3}\zeta\left(\tfrac n2\right)^{L+1}\ge 4\zeta n^{L+1}
    \]
    paths of length~$L$ in~$G[V']$, where we used~$|V'|\ge\tfrac n2$ and~$L+1\le 2r_1+1$. 
    Thus, Claim~\ref{clm:GoodNeighbour} entails that for any two distinct vertices~$u,v\in X$, there are that many paths of length~$L$ between~$N_G(v)\cap V_{i_v}'$ and~$N_G(u)\cap V_{i_u}'$. 
    At most~$2(L+1)n^{L}$ of these paths contain~$u$ or~$v$; appending~$u$ and~$v$ to each of the remaining ones, we obtain at least $4\zeta n^{L+1}-2(L+1)n^{L}\ge \zeta n^{L+1}$ paths from~$u$ to~$v$ with exactly~$L+1$ internal vertices, all of them in~$V'\subseteq X\setminus V(\cP)$. 
    Hence,
    \begin{equation}\label{eq:Gconn}
        \eqtext{for every pair of distinct vertices $u,v\in X$ there are at least $\zeta n^{L+1}$ paths from~$u$ to~$v$ with exactly $L+1$ internal vertices, all of them in $X\setminus V(\cP)$.}
    \end{equation}
    As~$|X\setminus V(\cP)|\le n$, this in particular says that~$G$ is~$(\zeta, X, X\setminus V(\cP), L+1)$-connected, and hence also~$(\zeta, X\setminus V(\cP),X\setminus V(\cP), L+1)$-connected.
    Similarly, for every~$v\in X$ there are at least~$4\zeta n^{L+1}$ paths of length~$L$ in~$G[V']$, ending and starting in~$N_G(v)\cap V_{i_v}'$. 
    At most~$(L+1)n^{L}$ of them contain~$v$, and each of the remaining ones forms, together with~$v$, a cycle of length~$L+2$ whose vertices other than~$v$ lie in~$V'\subseteq X\setminus V(\cP)$. 
    Since each such cycle arises in this way from at most two of these paths, we obtain at least~$\zeta n^{L+1}$ cycles of length~$L+2$ containing~$v$, as required. 
    
    Before applying Lemma~\ref{lem:abs}, let us comment on the choice of~$\gamma$. The parameters~$\ell$ and~$t$ with which we shall apply that lemma depend on~$L$, which is determined only at this point, whereas~$\gamma$ was fixed in~\eqref{eq:mainhierarchy}. Write~$\gamma'(\eps,\zeta,\ell,t)$ for the constant provided by Lemma~\ref{lem:abs}. Since~$L\le 2r_1\le 2T_1$, there are only finitely many possible values of the pair~$(\ell,t)=(\tfrac{L+1}{2},L+1)$, and we require in~\eqref{eq:mainhierarchy} that
    $$\gamma<\min\Big\{\gamma'\big(\zeta,\zeta,\tfrac{L+1}{2},L+1\big)\colon L\le 2T_1 \text{ odd}\Big\}\;;$$
    this is legitimate as the minimum on the right-hand side depends on~$\zeta$ and on~$T_1$ only, and~$\gamma$ stands to the left of~$\tfrac1{r_1}$ in~\eqref{eq:mainhierarchy}.

    Putting all this together we can apply Lemma~\ref{lem:abs}, with $X\setminus V(\cP)$ in place of~$U$, with $\zeta$ in place of $\eps$, and with $\ell:=\tfrac{L+1}{2}$ (which is an integer as~$L$ is odd) and $t:=L+1$. Let us make the matching of the parameters explicit. With this choice we have~$2\ell=L+1$ and~$2\ell+1=L+2$. Hence the hypothesis of Lemma~\ref{lem:abs} that every vertex of~$U$ lies in at least~$\eps n^{2\ell}$ copies of~$C_{2\ell+1}$ whose remaining~$2\ell$ vertices lie in~$U$ is exactly the statement about $\zeta n^{L+1}$ cycles of length~$L+2$ proved above, and the hypothesis that~$G$ is~$(\zeta,U,U,t)$-connected is~\eqref{eq:Gconn}. We obtain an~$(X\setminus V(\cP),\gamma^2)$-absorbing path~$\cQ$ with
    \begin{equation}\label{eq:sizqQ}
    |V(\cQ)|\le\gamma n\;.
    \end{equation}
    Lemma~\ref{lem:abs} provides~$\cQ$ with all its vertices in~$X\setminus V(\cP)$; in particular,
    \begin{equation}\label{eq:QinX}
    V(\cQ)\cap V(\cP)=\emptyset\;.
    \end{equation}
    
   \subsubsection*{Step 3: Choosing a reservoir for connections}
We begin with a claim that provides many connections between any two vertices in $X$.
\begin{claim}\label{cl:pairsconnection}
    For every pair of distinct vertices $u,v\in X$, there are at least $4\rho n$ internally disjoint paths from $u$ to $v$ of length $L+2$ whose internal vertices are in $X\setminus (V(\cP) \cup V(\cQ))$.
\end{claim}
\begin{claimproof}
Recall that $L\le 2r_1$ and therefore~$\rho\ll\gamma\ll \zeta\ll 1/L$ by~\eqref{eq:mainhierarchy}. Hence we get from~\eqref{eq:Gconn} and~\eqref{eq:sizqQ} that
\begin{equation}\label{eq:GGconn}
\eqtext{for every pair of distinct vertices $u,v\in X$ there are at least $\tfrac{\zeta}{2} n^{L+1}$ paths from~$u$ to~$v$ with exactly $L+1$ internal vertices, all of them in $X\setminus (V(\cP)\cup V(\cQ))$.}
\end{equation}
    Now, given vertices~$u,v\in X$, let~$\cB$ be the largest family of internally disjoint paths of length~$L+2$ from~$u$ to~$v$ whose internal vertices are in $X\setminus (V(\cP) \cup V(\cQ))$.
    Suppose for a contradiction that~$|\cB|<4\rho n$. 
    Since~$\cB$ is maximal, all other paths between~$u$ and~$v$ contain at least one vertex in~$V(\cB)$. 
    Since~$|V(\cB)|\le 4(L+3)\rho n$ the maximum number of paths connecting $u$ and $v$ and containing a vertex in~$V(\cB)$ is at most~$4(L+3)^2\rho n^{L+1}< \zeta n^{L+1}/2$. 
    This is a contradiction to~\eqref{eq:GGconn}. 
\end{claimproof}
The next claim provides a small reservoir set through which every pair of vertices from $X$ can be connected sufficiently many times.
\begin{claim}\label{cl:RRR} 
There exists a set $S\subseteq X\setminus (V(\cP) \cup V(\cQ))$ such that
\begin{enumerate}[label=(\alph*)]
    \item\label{en:BodA} $|S|\le 2\rho n$ and
    \item\label{en:BodB} for every pair of distinct vertices $u,v\in X$, there are at least $\rho^{L+2} n$ internally disjoint paths from $u$ to $v$ in $G$ whose internal vertices are contained in $S$.
\end{enumerate}
\end{claim}
\begin{claimproof}
The proof is using the probabilistic method. Take a random subset of vertices~$S\subseteq X\setminus (V(\cP) \cup V(\cQ))$ by independently including each vertex with probability~$\rho$.
By Markov's inequality we have that \ref{en:BodA} is satisfied with probability at least $1/2$.
For~\ref{en:BodB}, take a pair of vertices $u,v$ as above. Let $\mathcal{F}_{u,v}$ be a family of at least $4\rho n$ internally disjoint paths from $u$ to $v$, as provided by Claim~\ref{cl:pairsconnection}. The internal vertices of one path in $\mathcal{F}_{u,v}$ are entirely contained in $S$ with probability $\rho^{L+1}$. So, the expected number of paths from $\mathcal{F}_{u,v}$ whose internal vertices are entirely contained in $S$ is at least $4\rho n\cdot \rho^{L+1}$. 
Further, since those paths are internally disjoint, their containment events are independent. Hence, we can apply Lemma~\ref{lem:Chernoff}~\ref{it:Chernoff1}, and get that with probability at least $1-\exp(-\Theta(n))$, the number of paths from $\mathcal{F}_{u,v}$ whose internal vertices are entirely contained in $S$ is at least $\rho^{L+2}n$. 
We apply the union bound over the at most $n^2$ choices of pairs $(u,v)$ to obtain that~\ref{en:BodB} holds with probability at least $1-o(1)$. 
In total, $S$ satisfies all the required properties with positive probability and we can choose such an outcome of the random selection.
 \end{claimproof}

    \subsubsection*{Step 4: Constructing an almost spanning system of paths}
In this step we find a collection of paths that covers all but at most~$\gamma^2n$ many vertices from~$X$.
Throughout the proof, we need to treat the case when $V(\cP)$ is tiny and when it is substantial slightly differently. We call these two cases $(\heartsuit 1)$ and $(\heartsuit 2)$.
\begin{itemize}
    \item \emph{Case~$(\heartsuit 1)$:} We have $|V(\cP)|<2\eps_2 n$.\\
    Set~$G_0 := G[V\setminus (V(\cP)\cup V(\cQ)\cup S)]$. We apply Theorem~\ref{thm:reglem} to the graph~$G_0$ with the trivial prepartition $V(G_0)= \emptyset \sqcup V(G_0)$ and with error parameter~$\eps_2$. We obtain an~$\eps_2$-regular partition~$U_1,\dots, U_{r_2}$.
    \item \emph{Case~$(\heartsuit 2)$:} We have $|V(\cP)|\ge 2\eps_2 n$.\\
    Set~$G_0 := G[V\setminus (V(\cQ)\cup S)]$; note that~$V(\cP)\subseteq V(G_0)$ by~\eqref{eq:QinX} and by the choice of~$S$. 
    We apply Theorem~\ref{thm:reglem} to the graph~$G_0$ with the prepartition $V(G_0)= V(\cP) \sqcup (V(G_0)\setminus V(\cP))$ and with error parameter~$\eps_2$. We obtain an~$\eps_2$-regular partition~$U_1,\dots, U_{r_2}$.
\end{itemize}
Observe that in both $(\heartsuit 1)$ and $(\heartsuit 2)$, we have
    \begin{align}\label{eq:Y-partition}
        U_i\subseteq V(\cP) 
        \qquad\text{or}\qquad
        U_i\subseteq V\setminus (V(\cP)\cup V(\cQ)\cup S)
    \end{align}
    for every~$i\in [r_2]$.

Set $k:=\big\lfloor \tfrac12\min_{i\in[r_2]}|U_i|\big\rfloor$. Since the cluster sizes differ by a factor of at most $1+\eps_2$ (cf.~\ref{R:1}), we have for every~$i\in [r_2]$ that
\begin{equation}\label{eq:keven}
2k\le |U_i|\le 2k+\eps_2|U_i|+2\;.    
\end{equation}
    Let~$G_0'$ be obtained from~$G_0$ by removing all edges~$uv\in E(G_0)$ with
    \begin{itemize}
        \item $u,v\in U_i$ for some~$i\in [r_2]$, 
        \item $u\in U_i$,~$v\in U_j$, and~$(U_i,U_j)$ is not~$\eps_2$-regular, or
        \item $u\in U_i$,~$v\in U_j$, and~$(U_i,U_j)$ has density smaller than~$d_2$. 
    \end{itemize}
    Then by Theorem~\ref{thm:reglem} there is a set~$Z\subseteq V(G_0)$ with~$|Z|\le\eps_2 n$ such that for every vertex~$v\in V(G_0)\setminus Z$,
    \begin{align}\label{eq:deg-partition}
        \deg_{G_0}(v) < \deg_{G_0'}(v) + 2d_2n\;.
    \end{align}
    Let~$R_2$ be the~$(\eps_2,d_2)$-reduced graph of~$G_0$ with respect to the partition~$\{U_i\}_{i\in [r_2]}$; recall in particular that the weight of an edge~$ij\in E(R_2)$ is the density~$d_{G_0}(U_i,U_j)$, and that this is what the weighted degrees $\wdeg_{R_2}$ refer to.

    By~\ref{it:cutnorm} we have $\cutn{W_G-W}\le\eps_2/2$.    
    Next, we use Lemma~\ref{lem:delete} to find a graphon representation $W_{G_0}$ of $G_0$ with $\cutn{W_{G_0}-W_G}\le 5\gamma$, as follows. If we have~$(\heartsuit 2)$, then we use that~$|S\cup V(\cQ)|\le 2\gamma n $, and hence there exists a graphon representation $W_{G_0}$ such that $\cutn{W_{G_0}-W_G}\le 4\gamma$. If we have~$(\heartsuit 1)$, then we use that~$|V(\cP)\cup S\cup V(\cQ)|< 2\gamma n +2\eps_2 n$, and hence there exists a graphon representation $W_{G_0}$ such that $\cutn{W_{G_0}-W_G}\le 4\gamma+4\eps_2$.
    
    By Fact~\ref{fact:distredgraph} there exists a graphon representation $W_{R_2}$ such that $\cutn{W_{R_2}-W_{G_0}}\le d_2+10\eps_2$. Altogether this gives
    $$\cutn{W_{R_2}-W}\le\cutn{W_{R_2}-W_{G_0}}+\cutn{W_{G_0}-W_G}+\cutn{W_G-W}\le d_2+10\eps_2+5\gamma+\eps_2/2 <\eps_1<\tfrac{\beta\alpha^2}{8}\;.$$
    Hence, we can apply Proposition~\ref{prop:clustergraph} with~$\beta$ instead of~$\gamma$ and with~$C=\tfrac{1}{100\beta}$ to $R_2$ and $W$. 
    In particular, the subgraph~$R_2^* = R_2-\SMALLDEG_{R_2}(\alpha)$ satisfies~\ref{en:halfcov} and~\ref{en:connected} for every~$A_2\subseteq V(R_2^*)$ with~$|A_2|\le \alpha r_2/100$. Set
    $$A_2:=\{i\in V(R_2^*) \colon U_i\subseteq  V(\cP) \}\,.$$ 
    (In~$(\heartsuit 1)$, we of course have $A_2=\emptyset$.)
    Due to~\eqref{eq:Y-partition} if~$i\notin A_2$ then~$U_i\subseteq V\setminus (V(\cQ)\cup S\cup V(\cP))$.
    Since~$|U_i|\ge \tfrac{n}{2r_2}$ for every~$i$ by~\ref{R:1}, we have~$|A_2|\le \tfrac{2|V(\cP)|}{n}r_2 \leBy{\ref{it:proplowdeg}\ref{it:boundvrtP}} 4\eps_0\alpha r_2\le \tfrac{\alpha r_2}{100}$, meaning that the set $A_2$ satisfies the size condition of Proposition~\ref{prop:clustergraph}.
    Therefore, due to~\ref{en:halfcov},~$R_2^*-A_2$ is uniquely half-covered and Lemma~\ref{lem:OurCombOpt}\ref{en:OurCombOptPerfect} yields a half-integral perfect matching~$m\colon E(R_2^*-A_2)\to \{0,\frac12,1\}$. 
    Set~$U:=\bigcup_{i\in V(R_2^*)\setminus A_2}U_i$. In the claim below, we find an almost-spanning path system in $G_0[U]$.
    \begin{claim}\label{cl:pathsystem}
     There exists a collection~$\mathcal{F} = \{P_{i,j}\}_{m(ij)>0}$ of vertex-disjoint paths in~$G_0[U]$ consisting of at most~$r_2$ paths, covering all except at most~$\rho n$ many vertices from~$U$.
    \end{claim}
    \begin{claimproof}
    For every~$i\in V(R_2^*)\setminus A_2$ we find a partition of~$U_i=U_{(i,0)}\sqcup \bigsqcup_{j\in V(R_2^*)\setminus (A_2\cup \{i\})} U_{(i,j)}$, the following way, which depends on how $i$ is covered by $m$:
    \begin{enumerate}[label=(O\arabic*)]
        \item\label{en:MOne} If there exists $j^*\in V(R_2^*)\setminus (A_2\cup \{i\})$ such that $m(ij^*)=1$, then let $U_{(i,j^*)}\subseteq U_i$ be arbitrary of size $2k$, and for $j\in V(R_2^*)\setminus (A_2\cup \{i,j^*\})$, set $U_{(i,j)}:=\emptyset$. 
        \item\label{en:MHalf} Otherwise (recall that $m$ is perfect and half-integral), there exist two distinct indices $j_1,j_2\in V(R_2^*)\setminus (A_2\cup \{i\})$ such that $m(ij_1)=m(ij_2)=1/2$. Let $U_{(i,j_1)},U_{(i,j_2)}\subseteq U_i$ be arbitrary disjoint, each of size $k$, and set $U_{(i,j)}:=\emptyset$ for all other indices $j$.
    \end{enumerate}
In either case, define $U_{(i,0)}:=U_i\setminus \bigsqcup_{j\in V(R_2^*)\setminus (A_2\cup \{i\})} U_{(i,j)}$, and note by~\eqref{eq:keven} that 
\begin{equation}\label{eq:Ui0}
    |U_{(i,0)}|\le \eps_2|U_i|+2\le \rho|U_i|/2\;.
\end{equation}
Consider an arbitrary edge $ij\in E(R_2)$ such that $m(ij)>0$. Irrespective of whether the pair $(U_{(i,j)},U_{(j,i)})$ was obtained using~\ref{en:MOne} or~\ref{en:MHalf}, it spans at least one third of the pair $(U_i,U_j)$. Since the pair $(U_i,U_j)$ is~$\eps_2$-regular of density at least~$d_2$, by Fact~\ref{fact:regpairinherit}, the pair~$(U_{(i,j)},U_{(j,i)})$ is~$3\eps_2$-regular of density at least $d_2/2$.
By Fact~\ref{fact:almsppathinregp} there is a path~$P_{i,j}\subseteq G[U_{(i,j)},U_{(j,i)}]$ covering all but at most~$\rho |U_{(i,j)}|/2$ many vertices from~$U_{(i,j)}$ and all but at most $\rho |U_{(j,i)}|/2$ many vertices from~$U_{(j,i)}$.
Taking into account that also the vertices of $U_{(i,0)}$ are not covered by paths, we see using~\eqref{eq:Ui0} that our system of paths covers all but at most $\rho|U_i|$ vertices in each cluster $U_i$. Summing over all $i$, we obtain the bound for the uncovered vertices in Claim~\ref{cl:pathsystem}.
Finally, note that the number of edges $ij$ with $m(ij)>0$ is at most $r_2$. For each such edge, one path was added to the system. This gives the bound on the number of paths in Claim~\ref{cl:pathsystem}.
    \end{claimproof}

    Note that the vertices in~$\bigcup_{i\in A_2\cup \SMALLDEG_{R_2}(\alpha)} U_i$ are not covered by~$\mathcal{F}$. 
    We shall show that all but at most $2\sqrt{\eps_2}n$ of them lie in~$V(\cP)$. Note that in case~$(\heartsuit 1)$ the set $V(\cP)$ is disjoint from $V(G_0)$, so that this says that these clusters contain at most $2\sqrt{\eps_2}n$ vertices in total.
    First note that directly from the definition of~$A_2$ we have~$\bigcup_{i\in A_2} U_i \setminus V(\cP) =\emptyset$. 
    Secondly, consider an arbitrary $i\in \SMALLDEG_{R_2}(\alpha)$. Let us stress that the weighted degree of~$i$ in~$R_2$ controls only the \emph{average} degree of the vertices of~$U_i$, and so we have to argue that all but a few vertices of~$U_i$ are of small degree in~$G$.
    For $j\in N_{R_2}(i)$ write $d_{ij}:=d_{G_0}(U_i,U_j)$ and let $B_j:=\{v\in U_i\colon |N_{G_0}(v)\cap U_j|>(d_{ij}+\eps_2)|U_j|\}$. As the pair $(U_i,U_j)$ is $\eps_2$-regular, we have $|B_j|<\eps_2|U_i|$. Consequently, the set $D_i$ of those vertices of~$U_i$ which lie in $B_j$ for more than $\sqrt{\eps_2}r_2$ indices $j\in N_{R_2}(i)$ satisfies
    \begin{equation}\label{eq:Dismall}
    |D_i|\le \tfrac{1}{\sqrt{\eps_2}r_2}\sum_{j\in N_{R_2}(i)}|B_j|\le \sqrt{\eps_2}|U_i|\;.
    \end{equation}
    Now take a vertex~$v\in U_i\setminus (D_i\cup Z)$. Since the smallest cluster has at most $\tfrac{n}{r_2}$ vertices, \ref{R:1} gives $|U_j|\le (1+\eps_2)\tfrac{n}{r_2}$ for every~$j$. Together with $\sum_{j\in N_{R_2}(i)}d_{ij}=\wdeg_{R_2}(i)\le\alpha r_2$, this yields
    \begin{align*}
    \deg_{G_0'}(v)&\le \sum_{j\in N_{R_2}(i)}\big|N_{G_0}(v)\cap U_j\big|\le \big(\wdeg_{R_2}(i)+\eps_2 r_2+\sqrt{\eps_2}\,r_2\big)(1+\eps_2)\tfrac{n}{r_2}\\
    &\le \big(\alpha+\eps_2+\sqrt{\eps_2}\big)(1+\eps_2)n\leBy{\eqref{eq:mainhierarchy}}\tfrac54\alpha n\;.
    \end{align*}
    As at most $2\gamma n+2\eps_2n$ vertices of~$G$ are missing in~$G_0$, this and~\eqref{eq:deg-partition} give
    $$\deg_G(v) < \tfrac54\alpha n + 2d_2 n +2\gamma n+2\eps_2 n < 2\alpha n\,,$$
    which by \ref{it:proplowdeg}\ref{it:main1-lowdeginP} means that~$v$ was already covered by~$\cP$ and therefore~$v\in V(\cP)$.
    Hence, using~\eqref{eq:Dismall}, $|Z|\le\eps_2 n$, and $2\sqrt{\eps_2}\le\rho$,
    \begin{align} \label{eq:smallleftover}
    \begin{split}
              \big|V\setminus (V(\mathcal{F}) \cup V(\cP)\cup V(\cQ)\cup S)\big|
        =|U\setminus V(\mathcal{F})| + \bigg|\bigcup_{i\in A_2\cup\SMALLDEG_{R_2}(\alpha)} U_i\setminus V(\cP)\bigg|
        \le \rho n+\sqrt{\eps_2}n+\eps_2 n\le 2\rho n
        \;.
    \end{split}
    \end{align}
    In other words the path system~$\cP\cup\mathcal{F}\cup\{\cQ\}$ covers all of~$V\setminus S$ except for at most~$2\rho n$ many vertices, all of them lying in~$X\setminus V(\cP)$.
    
\subsubsection*{Step 5: Connecting all paths and completing the cycle}

    Finally, the path system~$\cP\cup\mathcal{F} \cup  \{ \cQ\}$ contains at most~$2/\alpha + r_2 +1\le 2r_2$ many paths.
    All these paths end in vertices in~$X$: the paths of~$\mathcal{F}$ have all their vertices in~$U\subseteq X\setminus V(\cP)$, the same holds for~$\cQ$ by~\eqref{eq:QinX}, and the ends of the paths of~$\cP$ are of degree at least~$2\alpha n$ by~\ref{it:proplowdeg}\ref{it:main1-highenddeg}. 
    Moreover by Claim~\ref{cl:RRR}\ref{en:BodB}, between every two of these vertices there are at least~$\rho^{L+2}n$ internally disjoint paths connecting them, with all internal vertices contained in~$S$. 
    As we make at most~$2r_2$ connections, each of which uses~$L+1$ vertices of~$S$, and as the paths above are internally disjoint, at each step at most~$2r_2(L+1)<\rho^{L+2}n$ of them are blocked by the connections made earlier. 
    Thus, we may connect all the paths one by one, into a single cycle~$C$, such that all new vertices are taken from~$S$. 
    In other words, $V(C) = V(\mathcal{F})\cup V(\cP)\cup V(\cQ)\cup S'$, where~$S'$ is a suitable subset of $S$. 
    Due to~\eqref{eq:smallleftover} and Claim~\ref{cl:RRR}\ref{en:BodA}, the number of vertices not covered by~$C$ is at most
    $$2\rho n + |S|\le 4\rho n \le \gamma^2 n\,.$$
    Denote the set of these uncovered vertices by~$B$ and recall that they all lie in~$X\setminus V(\cP)$. Since~$\cQ$ is an~$(X\setminus V(\cP),\gamma^2)$-absorbing path, we may replace the subpath~$\cQ$ of~$C$ by a path with the same ends and with vertex set~$V(\cQ)\cup B$. This turns~$C$ into a cycle with vertex set~$V(C)\cup B=V$, that is, into a Hamilton cycle of $G$.

\section{Concluding remarks}\label{sec:Concluding}

\subsection{Hamilton cycles with positive probability}

Theorem~\ref{thm:mainHC} identifies conditions~\ref{en:mainHC1}, \ref{en:mainHC2}, and \ref{en:mainHC3} on a graphon $W$ which are sufficient for $\G(n,W)$ to be Hamiltonian asymptotically almost surely. Furthermore, parts~\ref{en:mainHCnegative1}, \ref{en:mainHCnegative2}, and \ref{en:mainHCnegative3} demonstrate that these conditions are necessary. We can refine the negative cases. Specifically, we aim to distinguish between graphons for which $\G(n,W)$ is a.a.s.\ not Hamiltonian, and those where the property holds with probability bounded away from 0 and 1. 

Below, we list four conditions, each of which guarantees that $\G(n,W)$ is a.a.s.\ not Hamiltonian. We conjecture that this list characterizes all such strongly negative cases.

\begin{prop}\label{prop:mainHCneg}
Suppose that $W:\Omega^2\to[0,1]$ is a graphon. Then $\G(n,W)$ is a.a.s.\ not Hamiltonian if at least one of the following conditions holds:
\begin{enumerate}[label=(\roman*)]
    \item\label{en:mainHCStronglyNegative1} $W$ is not a connected graphon;
    \item\label{en:mainHCStronglyNegative2} $\lim_{\alpha\searrow 0}\frac{\mu(\SMALLDEG_W(\alpha))}\alpha=\infty$;
    \item\label{en:mainHCStronglyNegative3} $W$ has a narrow peninsula;
    \item\label{en:mainHCStronglyNegative4} there exists a partition $\Omega=S\sqcup T$ with $\mu(S)=\mu(T)=\frac12$ such that $W$ is zero almost everywhere on $(S\times S) \cup (T\times T)$.
\end{enumerate}
\end{prop}

\begin{conj}
The list in Proposition~\ref{prop:mainHCneg} is complete. That is, if a graphon $W$ does not satisfy any of the conditions~\ref{en:mainHCStronglyNegative1}--\ref{en:mainHCStronglyNegative4}, then 
\[
\limsup_{n\to\infty}\Probability[\text{$\G(n,W)$ is Hamiltonian}]>0.
\]
\end{conj}

We briefly sketch the justification for Proposition~\ref{prop:mainHCneg}. The validity of condition~\ref{en:mainHCStronglyNegative1} is discussed in Section~\ref{ssec:Negative1}. Condition~\ref{en:mainHCStronglyNegative2} follows from Theorem~3(a) in~\cite{ConnectivityPaper}, which establishes that such graphons yield disconnected random graphs a.a.s.

For condition~\ref{en:mainHCStronglyNegative3}, the argument is analogous to, but simpler than, the one presented in Section~\ref{ssec:Negative3}. Suppose $(A,B)$ is a witness that $W$ has a narrow peninsula (recall Definition~\ref{defi:peninsula}\ref{en:defiPen2}). By the Law of Large Numbers, after the vertex-generating stage~\ref{G1}, the partition of the vertex set $V=V_A\sqcup V_B\sqcup V_{\mathrm{rest}}$ (induced by types in $A$, $B$, or $\Omega\setminus (A\cup B)$) satisfies $|V_A|=(\mu(A)+o(1))n$ and $|V_B|=(\mu(B)+o(1))n$ a.a.s.
Since $W$ vanishes on $A \times (A \cup B)$, in the edge-generating step~\ref{G2}, no edges are inserted between $V_A$ and $V_A\cup V_B$ a.a.s.
Consider the function $f: V \to \{0, \frac{1}{2}, 1\}$ defined by
\[
f(v) = \begin{cases} 
0 & \text{if } v \in V_A, \\
\frac{1}{2} & \text{if } v \in V_B, \\
1 & \text{if } v \in V_{\mathrm{rest}}.
\end{cases}
\]
Obviously, this function $f$ constitutes a half-integral vertex cover. Its total weight is
\[
\|f\|_1 = \frac{1}{2}|V_B| + |V_{\mathrm{rest}}| = \left(\frac{1}{2}\mu(B) + 1 - \mu(A) - \mu(B) + o(1)\right)n\;.
\]
Using the peninsula parameters $\mu(B)=1-2a$ and $\mu(A)>a$, a straightforward calculation shows that $\|f\|_1 = (\frac{1}{2} - (\mu(A)-a) + o(1))n<\frac{n}2$. Consequently, the random graph does not contain a fractional perfect matching (and thus no Hamilton cycle) a.a.s.

Finally, regarding condition~\ref{en:mainHCStronglyNegative4}, note that $|V_S| \sim \mathrm{Bin}(n, 1/2)$ and $|V_T|=n-|V_S|$. For odd $n$ we cannot have $|V_S|=|V_T|$, while for even $n$ Stirling's formula gives $\Probability\left[|V_S|=|V_T|\right]=\binom{n}{n/2}2^{-n}=\Theta(n^{-1/2})$. Thus, a.a.s.\ one of the sets $V_S$ and $V_T$ has more than $\frac n2$ vertices, and a.a.s.\ both of them are independent in $\G(n,W)$. As no two vertices of an independent set are consecutive on a Hamilton cycle, a graph containing an independent set of more than half of its vertices is not Hamiltonian. Hence, $\G(n,W)$ is a.a.s.\ not Hamiltonian.

\subsection{Sparser inhomogeneous random graphs}
For any graphon $W \not\equiv 0$, the number of edges in $\G(n,W)$ is strongly concentrated around $\frac12\|W\|_1 \cdot n^2 = \Theta(n^2)$. Hence, $\G(n,W)$ serves as an inhomogeneous counterpart to the dense Erdős--Rényi random graph $\G(n,p)$ with a constant edge probability $p\in (0,1]$. We can consider sparser versions $\G(n,f(n)\cdot W)$ for some function $f(n)\to 0$. In such a setting, $W$ need not be bounded anymore. As the most prominent example, \cite{bollobas2007PhaseTransition} studies the giant component in inhomogeneous random graphs in the sparsest meaningful regime $f(n)=\frac{1}{n}$.

Given the threshold for Hamiltonicity in the homogeneous case obtained by Pósa and Korshunov, it is natural to consider the scaling $f(n)=\frac{\ln n}{n}$. 
The related problem of \emph{connectivity} was solved for this regime in~\cite{MR3252923} under two additional regularity assumptions, namely the continuity of $W$ and an integrability condition on the function $x\mapsto \int_{\Omega}W(x,y)^2\diff\mu(y)$. The former---and more notable---of these is likely a technical requirement of the proof technique rather than a fundamental structural necessity. Namely, \cite{MR3252923} shows that for $\G(n,\frac{\ln n}{n}\cdot W)$ to be a.a.s.\ connected, we need condition~\ref{C1} alongside a stricter degree requirement: the essential infimum of the degrees in $W$ must be strictly greater than~$1$. This elegantly contrasts with our condition~\ref{C2} from the dense regime, which only requires a light tail of small degrees. The necessity of this hard lower bound in the sparse setting stems directly from the threshold for the disappearance of isolated vertices: at the scaling $f(n)=\frac{\ln n}{n}$, a vertex whose type $x$ satisfies $\deg_W(x)<1$ stays isolated with probability roughly $n^{-\deg_W(x)}$, and so a set of types of positive measure and of $W$-degree less than~$1$ would a.a.s.\ yield isolated vertices. In analogy to Theorem~\ref{thm:mainHC}, we believe that if these two conditions from~\cite{MR3252923} hold and $W$ does not have a peninsula, then $\G(n,\frac{\ln n}{n}\cdot W)$ is a.a.s.\ Hamiltonian.

\subsection{Other spanning structures in $\G(n,W)$}
We can study the appearance of other spanning structures than a Hamilton cycle in $\G(n,W)$. Perhaps the closest one is a perfect matching, by which we mean---in order to avoid a parity issue---a matching covering at least $n-1$ vertices. Are all the conditions on $W$ in Theorem~\ref{thm:mainHC} necessary for $\G(n,W)$ to contain a.a.s.\ a perfect matching? Conditions~\ref{en:mainHC2} and~\ref{en:mainHC3} clearly are, as the features given in~\ref{en:mainHCnegative2} and~\ref{en:mainHCnegative3} are clear obstructions to the property of having a perfect matching, too. 

Less obviously, Condition~\ref{en:mainHC1} is also necessary. Indeed, if $W$ is not connected, then we can use a witness of its disconnectedness $\Omega=S\sqcup T$ to partition the vertices of $\G(n,W)$ into two types, $V=V_S\sqcup V_T$, and we know that each edge is contained entirely in $V_S$ or in $V_T$, a.a.s. The random variable $|V_S|$ has distribution $\mathrm{Bin}(n, \mu(S))$. It is easy to check that when $n$ is even, with probability $\frac12+o(1)$, both $|V_S|$ and $|V_T|$ are odd. Obviously, $\G(n,W)$ does not contain a perfect matching in that case.

The next natural spanning structure one might study is the square (or higher powers, say the $k$-th power) of a Hamilton cycle. Here, even the homogeneous (Erdős--Rényi) case is very complicated, with thresholds determined only partially (for squares of Hamilton cycles in~\cite{MR4273128}, up to a constant, and for $k\ge 4$ in~\cite{makai2025sharpthresholdshigherpowers}). In any counterpart of Theorem~\ref{thm:mainHC} for higher powers, we are likely to need to find the correct counterpart only to condition~\ref{en:mainHC3}. Indeed, it was shown in Theorem~4 in~\cite{ConnectivityPaper} that conditions~\ref{en:mainHC1} and~\ref{en:mainHC2} in fact guarantee higher vertex connectivity, which is necessary for the presence of a higher power of a Hamilton cycle. 

That is, while~\ref{en:mainHC3} stems from the fact that a Hamilton cycle contains a perfect matching (i.e., a perfect tiling by copies of $K_2$), the sought counterpart will be based on the fact that the $k$-th power of a Hamilton cycle contains a perfect tiling by copies of $K_{k+1}$. A \emph{tiling by copies of $K_{k+1}$} is a collection of vertex-disjoint copies of $K_{k+1}$. Paper~\cite{MR4186624} develops a graphon theory for tilings. Similarly to the graphon theory of matchings (recall the beginning of Section~\ref{ssec:fractionalvertexcoversgraphons}), the dual approach turns out to be cleaner. In particular, paper~\cite{MR4186624} introduces the notion of a \emph{$K_{k+1}$-fractional cover of $W$}, which is any measurable function $f:\Omega\to [0,1]$ such that for $\mu^{k+1}$-almost every $(k+1)$-tuple $(x_1,\dots,x_{k+1})\in \Omega^{k+1}$ we have that 
$ \sum_{i=1}^{k+1}f(x_i)\ge 1 $ or $ \prod_{i=1}^{k}\prod_{j=i+1}^{k+1}W(x_i,x_j)=0. $
Paper~\cite{MR4186624} shows that taking the infimum of $\|f\|_1$ over all $K_{k+1}$-fractional covers of $W$ leads to a parameter which corresponds to the proportional size of the largest tiling by $K_{k+1}$'s in graphs close to $W$ (in particular, Theorem~5.1 in~\cite{MR4186624} is relevant). Hence, our conjecture, whose case $k=1$ corresponds to Theorem~\ref{thm:mainHC}, is as follows.

\begin{conj}
Suppose that $W:\Omega^2\to[0,1]$ is a graphon, and $k\ge 2$ is fixed. 
Then $\G(n,W)$ contains a.a.s.\ the $k$-th power of a Hamilton cycle if and only if the following three conditions are fulfilled:
\begin{enumerate}[label=(\roman*)]
    \item $W$ is a connected graphon,
    \item we have $\lim_{\alpha\searrow 0}\frac{\mu(\SMALLDEG_W(\alpha))}\alpha=0$, and
    \item every $K_{k+1}$-fractional cover of $W$ of $L^1$-norm at most $\frac1{k+1}$ is equal to $\frac{1}{k+1}$ almost everywhere.
\end{enumerate}
\end{conj}

\section*{Acknowledgments}
The sketch of the beautiful peninsula was generated by Gemini. Claude Opus 5 was used in preparing this revision (between versions 1 and 2 on the arXiv). Its role was confined to checking proofs and editing: the mathematical ideas are those of version 1, although a number of errors were found and corrected in the process. The responsibility for the final text is entirely ours.

\bibliographystyle{plain}
\bibliography{references.bib}

\end{document}